\documentclass[11pt,reqno]{amsart}

\usepackage[all]{xy}
\usepackage{amsmath,amssymb,amsthm,amsfonts,enumerate,array,color,lscape,fancyhdr,layout,pst-all}
\usepackage[english]{babel}
\usepackage{young}
\usepackage{graphicx}
\usepackage{float}
\usepackage{pstricks}
\usepackage{orcidlink}

\newcounter{numberofremark}
\newcommand\nothing[1]{}
\usepackage[active]{srcltx}
\newcommand{\dcl}{\DeclareMathOperator}
\dcl\D{D}
\dcl\Hom{\textup{Hom}}
\dcl\im{\textup{Im}}
\dcl\p{\mathfrak{p}}
\dcl\proj{proj}
\dcl\Z{\mathbb{Z}}

\renewcommand\k{{\Bbbk}}

\newlength\yStones
\newlength\xStones
\newlength\xxStones

\makeatletter
\def\Stones{\pst@object{Stones}}
\def\Stones@i#1{%
  \pst@killglue%
  \begingroup%
  \use@par%
  \setlength\xxStones{\xStones}%
  \expandafter\Stones@ii#1,,\@nil
  \endgroup
  \global\addtolength\xStones{0.6cm}%
  \global\addtolength\yStones{-7.5mm}}%
\def\Stones@ii#1,#2,#3\@nil{%
  \rput(\xxStones,\yStones){%
    \psframebox[framesep=0]{%
      \parbox[c][6mm][c]{11mm}{\makebox[11mm]{$#1$}}}}%
  \addtolength\xxStones{1.2cm}%
  \ifx\relax#2\relax\else\Stones@ii#2,#3\@nil\fi}
\makeatother

\def\Stone#1{\fbox{\makebox[8mm]{\strut#1}}\kern2pt}

\newtheorem{theorem}{Theorem}[section]
\newtheorem{lemma}[theorem]{Lemma}
\newtheorem{corollary}[theorem]{Corollary}
\newtheorem{proposition}[theorem]{Proposition}
\newtheorem{example}[theorem]{Example}
\newtheorem{remark}[theorem]{Remark}

\newtheorem{definition}[theorem]{Definition}

\begin{document}
\allowdisplaybreaks

\title{Cochain complexes over an endofunctor}

\author[Germ\'an Benitez]{Germ\'an Benitez\, \orcidlink{0000-0003-0565-6300}}
\address{Centro de Estudos do Mar \\ Universidade Federal do Paraná \\ Pontal do Paraná \\ Brazil}
\email{germanbenitez@ufpr.br}
\address{Departamento de Matem\'atica \\ Universidade Federal do Amazonas \\ Manaus \\ Brazil}
\email{gabm@ufam.edu.br\,}

\author[Pedro Rizzo]{Pedro Rizzo\, \orcidlink{0000-0001-8793-0822}}
\address{\noindent Instituto de Matem\'aticas \\ Universidad de Antioquia \\ Medell\'in Ant \\ Colombia}
\email{pedro.hernandez@udea.edu.co\,(Corresponding author)}

%\date{\today}
%\makeatletter
%\def\@roman#1{\romannumeral#1}
%\makeatother

%===========================================================%
						%ABSTRACT%
%===========================================================%

\begin{abstract}
We introduce the categories of left and right complexes for (additive) endofunctors over an additive category $\mathcal{C}$. These categories unify the definitions of complexes over an additive category and modules over repetitive algebras. Supported by Eilenberg-Moore theory and Happel's successful work, we study the descriptions and properties of these categories, highlighting both similarities and differences with complexes and modules over repetitive algebras. We also examine the influence and interaction of these categories with the category of endofunctors, leading to applications in the representation theory of algebras and with eventual applications in ($2$-)algebraic geometry.
\end{abstract}

 

\maketitle

%===========================================================%
						%CONTENTS%
%===========================================================%

\tableofcontents    % imprime o sumário

%===========================================================%
					%INTRODUCTION%
%===========================================================%

\section{Introduction}

Let $A$ be a basic finite dimensional $\Bbbk$-algebra over a field $\Bbbk$. We denote by $A$-mod the finitely generated modules on $A$ and by $D(-)=\textup{Hom}_{\Bbbk}(-,\Bbbk)$ the duality functor on $A$-mod. The repetitive algebra $\widehat{A}$ of $A$, introduced by Hughes and Waschb\"usch in \cite{HW83}, plays a crucial role in the representation theory of algebras. We highlights some important properties of the repetitive algebra, such as: is an selfinjective infinite dimensional algebra without unity; $\widehat{A}$ is a special biserial algebra if and only if $A$ is a gentle algebra  (see \cite{Sch}). Additionally, the category of finitely generated left modules over $\widehat{A}$, denoted by $\widehat{A}-$mod, is a Frobenius category (see \cite{H88}). 

In his fundamental book \cite{H88}, Happel established a key connection between the bounded derived category $D^b(A)$ and the stable category $\widehat{A}\textup{-\underline{mod}}$ via a fully faithful triangulated embedding. This embedding becomes an equivalence of triangulated categories if the algebra $A$ has finite global dimension. The result, the methods used in its proof, and the developments that followed have had a significant impact on the representation theory of algebras.

In algebraic geometry, based on earlier work by
Bernstein--Gelfand--Gelfand \cite{BGG} and Beilinson \cite{Beil}, Dowbor and Meltzer in \cite{DM1} established a significant triangulated equivalence between the category of bounded derived category $D^b(\textup{Coh}(\mathbb{P}_{\scriptstyle\mathbb{C}}^n))$ of coherent sheaves over the projective $n$-space , and $\widehat{\Lambda}_n\textup{-\underline{mod}}$. Here $\Lambda_n$, represents the (finite dimensional) exterior algebra of the space $\mathbb{C}^{n+1}$. This equivalence has implications and applications in several areas such as the results in \cite{DM2} and \cite{FGRV}.

These equivalences with bounded derived categories illustrated one part of the importance of repetitive algebras. However, their applications and theoretical implications extend to an ample set of problems where they are key tools. For example, repetitive algebras are used to classify self-injective algebras via Galois coverings (\cite{SY2}). They are also used to study Nakajima varieties, which are geometric objects related to perverse sheaves over quantum loop algebras, by using repetitive algebras of Dynkin quivers (\cite{LP}). Finally, the equivalences between derived categories of an algebra can be generalized to equivalences in the stable category of repetitive modules over an algebra via Wakamatsu-tilting modules (\cite{w21}).

The preceding discussion highlights the significance of repetitive algebras, their modules, and their derived properties. This overview of the impact in the field was our initial motivation for exploring and generalizing these concepts. While specific cases of generalizations and equivalences have been addressed by researchers such as Keller \cite{K} and Asashiba and Nakashima \cite{AN}, this paper presents a different generalization using a new theoretical framework.

More precisely, in this work we introduce the categories $\mathfrak{L}^*_F(\mathcal{C})$ and $\mathfrak{R}^*_F(\mathcal{C})$ for an endofunctor $F\in\textup{End}(\mathcal{C})$, which unify the definitions of the category of complexes over an additive category and the category of modules over repetitive algebras. Inspired by their close connection to graded categories, our categories can be identified as subcategories of an Eilenberg-Moore category over an appropriate monad (see Section \ref{sec:F-chain}). This framework allows us to study their categorical properties and provide concrete constructions for their main objects and a systematic study of them (see Sections \ref{sec:propert} and \ref{sec:proj_inj}).

Inspired by the successful constructions for repetitive algebras from Happel's work in \cite{H88}, we formalize some of his assertions. This establishes a new framework and defines relationships between our categories related to constructions in $\widehat{A}$-modules. Additionally, we also compare our categories with the category of complexes over an abelian category (see Section \ref{sec:LF(C)vsEnd(C)}). These conclusions open the door for exploring new properties of endofunctors of a given category and finding characterizations that allow us to study what we call generalized repetitive algebras within this new context. In fact, in Section \ref{sec:apply}, we explore two main applications of our constructions. First, motivated by alternative constructions of repetitive algebras found in the literature that use different modules for tensorization (\cite{Ass18, ABS09}), we analyze the conditions under which an endofunctor is a tensorization. We also study the importance of the object used for such tensorization. This helps to explain why the dual $D(A)$ is chosen as the tensorizing object in repetitive algebra and whether this choice is involved in the categorical properties of the category of modules over $\widehat{A}$.  Secondly, we use the characterization results from the first part of this section, along with the characterization of endofunctors over the category of quasi-coherent modules on a specific type of geometric scheme (see \cite{BC}). This allows us to extend our results to a  ``global'' context from a geometrical perspective (in the commutative case). The results in this section have potential applications in $2$-algebraic geometry, since categories $\mathfrak{L}_F(\textup{\textbf{QCoh}}(X))$ are natural generalizations of categories of complexes, which are key in this area.

Although our work presents new tools and constructions that align more closely with recent work in the field, such as \cite{HJ}, \cite{Gaoetal}, it is worth acknowledging that similar ideas were explored in specific cases by previous authors. In \cite{FGR75} the authors defined the trivial extension of an abelian category by an endofunctor as a generalization of the category of modules over a trivial extension of a ring. They established various properties and studied their homological implications. This approach recently appeared in \cite{Mao23}, where the author explored in more detail projective and injective objects and their applications in Morita context rings. However, both these works differ significantly from our focus, results, and applications.

\subsection*{Roadmap and overview of the results.}
This article is organized as follows. In Section~\ref{sec:F-chain}, we introduce the categories of left (resp. right) cochain complexes over an endofunctor $F:\mathcal{C}\longrightarrow\mathcal{C}$, denoted as $\mathfrak{L}_F(\mathcal{C})$ (resp. $\mathfrak{R}_F(\mathcal{C})$). As shown in Example \ref{example:L_F(C)} we can recover well-known classes of categories by choosing different functors $F$. The main result of this section is Theorem \ref{thm:graduate}, which characterizes these new categories as a special class of subcategories of Eilenberg-Moore categories on a suitable monad.  In Section~\ref{sec:propert}, we study the expected properties of these categories, highlighting their similarities and differences with the standard category of cochain complexes. We also analyze how some properties depend on the choice of $F$.  In fact, Theorems \ref{thm:abelian} and \ref{thm:prod and eq} present a detailed collection of properties, some of which are derived from the previous characterization and the established theory of Eilenberg-Moore categories.  Inspired by the basic examples of our constructions, such as repetitive algebras and cochain complexes Section~\ref{sec:proj_inj} is dedicated to finding indecomposable projective and injective objects in $\mathfrak{L}^*_F(\mathcal{C})$ and $\mathfrak{R}^*_F(\mathcal{C})$ (see Theorem \ref{proj&inj}). Our goal is to understand, for instance, when these categories are Frobenius categories. In Section \ref{sec:LF(C)vsEnd(C)}, we explore the relationship between the category of endofunctors and our categories of complexes. This study yields a functorial correspondence (see Theorem \ref{thm:LR}) that is an embedding but does not reflect isomorphisms (see Remark \ref{rmk:reflectiso}). We also compare our constructions to the standard category of complexes and to some of the main ideas from Happel's influential theory of repetitive algebras in \cite{H88}. These two sets of results are presented in Proposition \ref{prop:complexequiv} and Theorem \ref{thm:adjoint}, respectively. Inspired by these connections, in Section \ref{sec:apply}, we explore applications of our work. We present new characterization results and their consequences, seeking to generalize properties and results to what we call ``category of modules over generalized repetitive algebras''. For these categories, we present a sequence of relations and results in Subsection \ref{subs:locfinite}, generalizing, for instance, the correspondence from \cite{Gir18} in Proposition \ref{B[A]functors1}. Finally, we extend the results of the first subsection to a more general, geometric context in Subsection \ref{geom_app}.

%===============================================%
				%PRELIMINARIES%
%===============================================%

\section{Cochain complexes over an endofunctor}
\label{sec:F-chain}

Let $\mathcal{C}$ be an additive category, and let $F:\mathcal{C}\longrightarrow\mathcal{C}$ be an additive covariant functor. Denote by $\mathfrak{L}_F(\mathcal{C})$ the category with objects defined to be a family $M=\left(M^n,d^n_M\right)_{n\in\mathbb{Z}}$ of objects $M^n$ in $\mathcal{C}$ and morphisms $d_M^n:F(M^n)\longrightarrow M^{n+1}$ in $\mathcal{C}$ satisfying $d^{n+1}_MF(d^n_M)=0$ for all $n\in\mathbb{Z}$. This condition guarantees that the following composition in $\mathcal{C}$ vanishes for all  $n\in\mathbb{Z}$:
	$$
	\xymatrix{
	F^2(M^n)\ar[r]^-{F\left(d^n_M\right)} & F(M^{n+1})\ar[r]^-{d^{n+1}_M} & M^{n+2}.
	}
	$$
We will call these objects $F$-\emph{cochain complexes} or \emph{cochain complexes over $F$}. We can visualize an $F-$cochain complex diagrammatically as
	$$
	\xymatrix{
	M:=\cdots\ar@{~>}[r] & M^{n-1}\ar@{~>}[r]^-{d^{n-1}_M} & M^n\ar@{~>}[r]^-{d^n_M} & M^{n+1}\ar@{~>}[r]^-{d^{n+1}_M} & M^{n+2}\ar@{~>}[r] & \cdots}
	$$
A morphism $\varphi:M\longrightarrow N$ in $\mathfrak{L}_F(\mathcal{C})$ is then a family $\varphi=(\varphi^n)_{n\in\mathbb{Z}}$ of morphisms $\varphi^n:M^n\longrightarrow N^n$ in $\mathcal{C}$ satisfying $\varphi^{n+1}d^n_M=d^n_NF(\varphi^n)$ for all $n\in\Z$. This condition ensures that the following diagram commutes in $\mathcal{C}$ for all $n\in\Z$
	$$
	\xymatrix{
	F(M^n)\ar[r]^-{d^n_M}\ar[d]_{F(\varphi^n)} & M^{n+1}\ar[d]^{\varphi^{n+1}} \\
	F(N^{n})\ar[r]_-{d^n_N}  & N^{n+1}
	}
	$$
Such a morphism can be seen diagrammatically as
	$$
	\xymatrix{	
	M=\ar@<-13pt>[d]_{\varphi}\cdots\ar@{~>}[r] & M^{n-1}\ar@{~>}[r]^-{d^{n-1}_M}\ar[d]_{\varphi^{n-1}} & M^n\ar@{~>}[r]^-{d^n_M} \ar[d]_{\varphi^n} & M^{n+1}\ar@{~>}[r]^-{d^{n+1}_M}\ar[d]_{\varphi^{n+1}} & M^{n+2}\ar@{~>}[r]\ar[d]_{\varphi^{n+2}} & \cdots\\
	N=\cdots\ar@{~>}[r] 						& N^{n-1}\ar@{~>}[r]_-{d^{n-1}_N}                       & N^{n}\ar@{~>}[r]_-{d^n_N}                   & N^{n+1}\ar@{~>}[r]_-{d^{n+1}_N}        				  & N^{n+2}\ar@{~>}[r] & \cdots}
	$$
We introduce a full subcategory $\mathfrak{L}^b_F(\mathcal{C})$ within $\mathfrak{L}_F(\mathcal{C})$, consisting of objects $M=\left(M^n,d^n_M\right)_{n\in\mathbb{Z}}$ with the property that almost all $M^n$ are $0$. We will denoted both categories briefly by $\mathfrak{L}^*_F(\mathcal{C})$.

%-----------------------------------------------------------%

\begin{example}
\label{example:L_F(C)}
\noindent

\begin{enumerate}[(i)]
\item $\mathfrak{L}^*_F(\mathcal{C})$ coincides with the category of cochain complexes of $\mathcal{C}$, when $F=\textup{Id}_{\mathcal{C}}$ is the identity functor.

\item When $F=0$ is the zero functor, the category $\mathfrak{L}^*_0(\mathcal{C})$ is the category of complexes where all differentials are zero. As a consequence of the proof of Theorem \ref{thm:graduate} below, $\mathfrak{L}^*_0(\mathcal{C})$ is equivalent to the category $\mathfrak{gr}_0^*(\mathcal{C})$, which is the category of $\mathbb{Z}$-graded objects in $\mathcal{C}$. We will revisit this category in Remark \ref{rem:falgebras}.

\item Let $A$ be a $\Bbbk$-algebra over a field $\Bbbk$. We denote by $A\textup{-mod}$ the category of finitely generated left $A$-modules. Futhermore, $D(-)=\Hom_{\Bbbk}(-,\Bbbk)$ represents the standard duality on $A\textup{-mod}$. Additionally, $\widehat{A}$ denotes the repetitive algebra of $A$,  introduced by Hughes and Waschb\"usch in \cite{HW83}. According to the authors in \cite{HW83}, the category $\mathfrak{L}^b_{D(A)\otimes_A -}\left(A\textup{-mod}\right)$ coincides with the ca\-te\-go\-ry of $\widehat{A}\textup{-mod}$. Subsequently, Giraldo in \cite{Gir18}, established that these two categories are equivalent.
\end{enumerate}
\end{example}

%-----------------------------------------------------------%

\begin{remark}\rm
\label{rem}
\noindent

\begin{enumerate}[(i)]
\item In a similar manner, we define categories $\mathfrak{R}_F(\mathcal{C})$ and $\mathfrak{R}^b_F(\mathcal{C})$ (shortly, $\mathfrak{R}^*_F(\mathcal{C})$) whose objects and morphisms, analogous to those in $\mathfrak{L}_F(\mathcal{C})$, possess the following structure: Objects are the families $M=\left(M^n,d^n_M\right)_{n\in\mathbb{Z}}$, with $M^n\in \mathcal{C}$, and morphisms $d_M^n:M^n\longrightarrow F(M^{n+1})$ in $\mathcal{C}$ satisfying $F(d_M^{n+1})d_M^n=0$, holds for all $n\in\Z$. The full subcategory $\mathfrak{R}^b_F(\mathcal{C})$ of $\mathfrak{R}_F(\mathcal{C})$ is defined by objects $M=\left(M^n,d^n_M\right)_{n\in\mathbb{Z}}$ where almost all $M^n$ are zero.

\item Similar to the cochain and chain complex for an additive category, we can define left and right chain complexes over an endofunctor on any additive category.
\end{enumerate}

\end{remark}

%-----------------------------------------------------------------------------------%
Our first result provides an interesting characterization of the categories $\mathfrak{L}^{\ast}_F(\mathcal{C})$ as subcategories of Eilenberg-Moore categories, achieved through a graded identification. To this end, in the proof of Theorem \ref{thm:graduate}, we introduce the category $\mathfrak{gr}_F(\mathcal{C})$. The objects of  $\mathfrak{gr}_F(\mathcal{C})$ are $\mathbb{Z}$-graded objects $M=\bigoplus\limits_{n\in\mathbb{Z}}M_n$ in $\mathcal{C}$ equipped with a morphism $d_M:F(M)\longrightarrow M$. This morphism must satisfy two conditions: $d_MF(d_M)=0$ and $d_M(F(M_n))\subseteq M_{n+1}$ (i.e., $d_M(F(M_n))$ is a subobject of $M_{n+1}$) for all $n\in\mathbb{Z}$. We denote such an object by $(M,d_M)$. A morphism in $\mathfrak{gr}_F(\mathcal{C})$ is a degree-preserving morphism $\varphi:(M,d_M)\longrightarrow (N,d_N)$ in $\mathcal{C}$ that satisfies $\varphi d_M=d_NF(\varphi)$, i.e, that makes the following diagram commute in $\mathcal{C}$:
$$
\xymatrix{	
F(M)\ar[r]^-{d_M}\ar[d]_{F(\varphi)} & M\ar[d]^{\varphi} \\
F(N)\ar[r]_-{d_N}  & N.
}
$$
Denote by $\mathfrak{gr}^b_F(\mathcal{C})$, the full subcategory of $\mathfrak{gr}_F(\mathcal{C})$, whose objects $M=\bigoplus\limits_{n\in\mathbb{Z}}M_n$ have the property that almost all $M_n$ are zero.

%--------------------------------%
\begin{example}
Let $\mathcal{C}$ be the category of $\Bbbk$-vector spaces, $A$ be a $\Bbbk$-algebra and $F=A\otimes_{\Bbbk}-$ be an endofunctor over $\mathcal{C}$. Thus, $\mathfrak{gr}_F(\mathcal{C})$ is the category of  $\mathbb{Z}$-graded $A$-modules $M=\bigoplus\limits_{n\in\Z}M_n$ such that 
	$$
	A^2M=0\ \ \text{and }\ AM_n\subseteq M_{n+1}\ \ \text{for all $n\in\mathbb{Z}$.}
	$$
Modules with the first property are well-known as \emph{$2$-nilpontent modules}. Note that this condition is automatically satisfied for \emph{zero-algebras} or \emph{algebras with zero square radical}.
\end{example}
	
%---------------------------------------%

% Denote by $\mathfrak{gr}^b_F(\mathcal{C})$, the full subcategory of $\mathfrak{gr}_F(\mathcal{C})$, whose objects  $M=\left(M^n,d^n_M\right)_{n\in\mathbb{Z}}$ has the property that almost all $M^n$ are zero.

%----------------------------------%
% \begin{theorem}\label{thm:graduate}
% Let $\mathcal{C}$ be an additive category and let $F:\mathcal{C}\longrightarrow\mathcal{C}$ be a covariant functor. Under the above notations, the categories $\mathfrak{L}_F(\mathcal{C})$ and $\mathfrak{gr}_F(\mathcal{C})$ are isomorphic. Moreover, the categories $\mathfrak{L}^b_F(\mathcal{C})$ and $\mathfrak{gr}^b_F(\mathcal{C})$ are isomorphic.
% \end{theorem}

%----------------------------------------%
\begin{theorem}\label{thm:graduate}
Let $\mathcal{C}$ be an additive category and $F:\mathcal{C}\longrightarrow\mathcal{C}$ a covariant functor. The category $\mathfrak{L}_F(\mathcal{C})$  is isomorphic to a subcategory of the Eilenberg-Moore category $\mathcal{C}^{T}$ associated with the monad $(T,\eta,\mu)$, where the endofunctor $T:\mathcal{C}\longrightarrow\mathcal{C}$ is defined by
    $$
    \begin{array}{rccl}
    T: & \mathcal{C} & \longrightarrow & \mathcal{C}\\
       & M &\longmapsto & T(M):=M\oplus F(M)\\
    & \varphi &\longmapsto & T(\varphi):=\left[\begin{array}{cc}
        \varphi & 0 \\
        0       & F(\varphi)
        \end{array}\right]
    \end{array}
    $$
and the natural transformations $\eta:\textup{Id}_{\mathcal{C}}\longrightarrow T$ and $\mu:T^2\longrightarrow T$  are given by
    $$
    \eta_M:=\left[\begin{array}{c}
        1_M \\
        0   
        \end{array}\right],
    \ \ \ \ \ 
    \text{and}
    \ \ \ \ \ 
    \mu_M:=\left[\begin{array}{cccc}
        1_M & 0 & 0 & 0 \\
        0   & 1_{F(M)} & 1_{F(M)} & 0
        \end{array}\right],\ \ \ \ \ \text{for all $M\in\mathcal{C}$.}
    $$
\end{theorem}
\begin{proof}
The proof is divided into two main steps.\\
{\bf Step 1: The isomorphism between the categories $\mathfrak{L}_F(\mathcal{C})$ and $\mathfrak{gr}_F(\mathcal{C})$.}\\ 
For this, we consider the following covariant functors: 
    $$
\begin{array}{cclcccccl}
\mathcal{F}: \mathfrak{L}_F(\mathcal{C})&\longrightarrow & \mathfrak{gr}_F(\mathcal{C})&&&&\mathcal{G}:\mathfrak{gr}_F(\mathcal{C})&\longrightarrow & \mathfrak{L}_F(\mathcal{C})\\
\ \ \ M &\longmapsto & \mathcal{F}(M):=\left(\bigoplus\limits_{n\in\mathbb{Z}}M^n,\bigoplus\limits_{n\in\mathbb{Z}}d^n_M\right)&&&& \ \ \ \left(M,d_M\right) &\longmapsto & \mathcal{G}(\left(M,d_M\right)):=\left(M_n, d^n_M\right)_{n\in\mathbb{Z}}\\
 \ \ \ \varphi &\longmapsto & \mathcal{F}(\varphi):=\bigoplus\limits_{n\in\mathbb{Z}}\varphi^n&&&& \ \ \ \varphi &\longmapsto & \mathcal{G}(\varphi):=\left(\varphi^n\right)_{n\in\mathbb{Z}}
    \end{array}
    $$
where, for the functor $\mathcal{G}$, we write $d^n_M=\pi_{n+1}d_MF(\iota_n)$ and $\varphi^n=\pi_{n}\varphi \iota_n$, with the canonical inclusion $\iota_n:M_n\longrightarrow M$ and canonical projection $\pi_n:M\longrightarrow M_n$ for all $n\in\mathbb{Z}$. Follows from these covariant functors that 
	$$
	\mathcal{FG}=\textup{Id}_{\mathfrak{gr}_F(\mathcal{C})}\ \ \mbox{and}\ \ \mathcal{GF}=\textup{Id}_{\mathfrak{L}_F(\mathcal{C})},
	$$
which proves the isomorphism between these two categories.\\
\noindent
{\bf Step 2: $\mathfrak{gr}_F(\mathcal{C})$ coincides with a subcategory of $\mathcal{C}^{T}$.}\\ 
Indeed, it is straightforward to verify (following, e.g., \cite[p. 136]{Mac71} or \cite[p. 259]{Awo10}) that $\mathfrak{gr}_F(\mathcal{C})$ coincides with the full subcategory of Eilenberg-Moore category $\mathcal{C}^{T}$ whose objects are $T$-algebras $(M,\alpha)$ that satisfy two conditions:
\begin{itemize}
\item $M$ is $\mathbb{Z}$-graded in $\mathcal{C}$; that means, $M=\bigoplus\limits_{n\in\mathbb{Z}}M_n$.
\item $d(F(M_n))\subseteq M_{n+1}$, where $d=\alpha\pi:F(M)\longrightarrow M$ and $\pi:T(M)=M\oplus F(M)\longrightarrow F(M)$ is the canonical projection.  
\end{itemize} 
\end{proof}

%----------------------------------------------------------------------------------%

\begin{remark}\label{rmk:EM}\rm
\begin{enumerate}
\item A direct consequence of the proof of Theorem \ref{thm:graduate} is that the categories $\mathfrak{L}^b_F(\mathcal{C})$ and $\mathfrak{gr}^b_F(\mathcal{C})$ are isomorphic.
\item The identification of  $\mathfrak{L}_F(\mathcal{C})$ as a subcategory of a particular Eilenberg-Moore category, as established by Theorem \ref{thm:graduate}, implies that certain hypotheses and conclusions concerning the properties of $\mathfrak{L}^{\ast}_F(\mathcal{C})$ in Section \ref{sec:propert} are adaptations of established results on Eilenberg-Moore categories (see Theorem \ref{thm:prod and eq} and Corollary \ref{thm:completness}).
\end{enumerate}
\end{remark}

%-----------------------------------------------%

%========================================%
					%PROPERTIES%
%============+===========================%
				
\section{Some properties}% 
\label{sec:propert}		
Motivated by the coincidence of $\mathfrak{L}^*_F(\mathcal{C})$ with cochain complexes when $F$ is the identity functor (cf. Example \ref{example:L_F(C)}(i)), in this section we will examine some of its properties to exhibit both similarities and differences compared to cochain complexes in $\mathcal{C}$. Obviously, this study involves results with consequences over the categories $\mathfrak{L}^*_F(\mathcal{C})$ and $\mathfrak{R}^*_F(\mathcal{C})$. We introduce the following notation: For any category $\mathcal{C}$, we denote $\textup{End}(\mathcal{C})$ the category of endofunctors $F:\mathcal{C}\longrightarrow \mathcal{C}$.

From now on, in this paper we restrict our focus to the realm of additive categories and additive covariant functors.

%--------------------------------%

\begin{theorem}
\label{thm:abelian}
Let $F:\mathcal{C}\longrightarrow \mathcal{C}$ be an endofunctor on the category $\mathcal{C}$.  

\begin{enumerate}[(i)]

\item The categories $\mathfrak{L}^*_F(\mathcal{C})$ and $\mathfrak{R}^*_F(\mathcal{C})$ are additive categories. 

\item If $\mathcal{C}$ has kernels, then $\mathfrak{L}^*_F(\mathcal{C})$ has kernels.

\item If $\mathcal{C}$ is an abelian category and $F$ preserves cokernels, then $\mathfrak{L}^*_F(\mathcal{C})$ is an abelian category.

\item If $\mathcal{C}$ has cokernels, then  $\mathfrak{R}^*_F(\mathcal{C})$ is also closed by cokernels.

\item If $\mathcal{C}$ is an abelian category and $F$ preserves kernels, then $\mathfrak{R}^*_F(\mathcal{C})$ is an  abelian category.

\item If $\mathcal{C}$ is a $\Bbbk$-category, where $\Bbbk$ is a commutative ring with unity, then  $\mathfrak{L}^*_F(\mathcal{C})$ and $\mathfrak{R}^*_F(\mathcal{C})$ are $\Bbbk$-categories.
\end{enumerate}
\end{theorem}

%-----------------------------------------------------------%

\begin{proof}
\noindent

\begin{enumerate}[(i)]
\item The proof follows the same line of reasoning as for cochain complexes.
\item Like cochain complexes, the kernel of a morphism $\varphi:M\longrightarrow N$ in $\mathfrak{L}^*_F(\mathcal{C})$ is given by a pair $\left(K,\iota\right)$, where $\iota=(\iota^n)_{n\in\mathbb{Z}}$ is determined by the canonical inclusions $\iota^n:\ker(\varphi^n)\longrightarrow  M^n$ and $K=\left(\ker(\varphi^n),d_{K}^n\right)_{n\in\mathbb{Z}}$ is defined by the unique morphism $d_{K}^n:F(\ker(\varphi^n))\longrightarrow \ker(\varphi^{n+1})$ such that
    \begin{equation}
    \label{eq:iotaker}
    d_M^{n}F(\iota^n)=\iota^{n+1}d_{K}^n.
    \end{equation}
The existence of these unique $d_{K}^n$ is guaranteed by the universal property of kernels. To see this, first observe that $\varphi^{n+1}d_M^{n}F(\iota^n)=d_N^nF(\varphi^n)F(\iota^n)=d_N^nF(\varphi^n\iota^n)=0$. Since $\iota^{n+1}$ is a monomorphism and $\iota^{n+1}d_K^nF(d_K^{n-1})=d_M^nF(\iota^n)F(d_K^{n-1})=d_M^nF(d_M^{n-1})F^2(\iota^{n-1})=0$, we obtain that $d_K^nF(d_K^{n-1})=0$. Now, we will prove the universal property of kernel $\left(K,\iota\right)$. For this, we consider another morphism $\eta:L\longrightarrow M$ in $\mathfrak{L}^*_F(\mathcal{C})$ such that $\varphi\eta=0$. This implies that $\eta^n:L^n\longrightarrow M^n$ is a morphism in $\mathcal{C}$ such that $\varphi^n\eta^n=0$ and     
    \begin{equation}
    \label{eq:etaPUker}
    \eta^{n+1}d_L^n=d_M^nF(\eta^n) \ \ \ \text{for all $n\in\mathbb{Z}$,}
    \end{equation} 
from universal property of kernel in $\mathcal{C}$, for each $n\in\mathbb{Z}$, there exists a unique $\omega^n:L^n\longrightarrow\ker(\varphi^n)$ such that
    \begin{equation}
    \label{eq:etaker}
    \eta^n=\iota^n\omega^n.
    \end{equation} 
Consider the family $\omega=\left(\omega^n\right)_{n\in\mathbb{Z}}$. To see that $\omega\in\textup{Hom}_{\mathfrak{L}^*_F(\mathcal{C})}(L,K)$, we need to show that 
    $$
    \omega^{n+1}d_L^n=d_K^nF(\omega^n),\ \ \ \text{for all $n\in\mathbb{Z}$}.
    $$
Indeed, by equations \eqref{eq:iotaker}, \eqref{eq:etaPUker} and \eqref{eq:etaker} we have $\iota^{n+1}d_K^nF(\omega^n)=d_M^{n}F(\iota^n)F(\omega^n)=d_M^{n}F(\eta^n)=\eta^{n+1}d_L^n=\iota^{n+1}\omega^{n+1}d_L^n$, and due to the fact that $\iota^{n+1}$ is a monomorphism, we conclude that $d_K^nF(\omega^n)=\omega^{n+1}d_L^n$, which completes the proof of item $(ii)$.

\item From item $(ii)$ we know that $\mathfrak{L}^*_F(\mathcal{C})$ has kernels. Now, for any morphism $\varphi:M\longrightarrow N$ in $\mathfrak{L}^*_F(\mathcal{C})$ its cokernel  is $\left(\textup{coker}(\varphi),\pi\right)$, where $\pi=(\pi^n)_{n\in\mathbb{Z}}$ is defined by the canonical projections $\pi^n:N^n\longrightarrow \textup{coker}(\varphi^n)$ and $\textup{coker}(\varphi)=\left(\textup{coker}(\varphi^n),d_{\textup{coker}(\varphi)}^n\right)_{n\in\mathbb{Z}}$. Here, since $F$ commutes with cokernels, the morphism $d_{\textup{coker}(\varphi)}^n:F(\textup{coker}(\varphi^n))=\textup{coker}(F(\varphi^n))\longrightarrow \textup{coker}(\varphi^{n+1})$ corresponds to (the unique) morphism with the property $\pi^{n+1}d_N^{n}=d_{\textup{coker}(\varphi)}^nF(\pi^{n})$. The existence of these $d_{\textup{coker}(\varphi)}^n$ is guaranteed by universal property of cokernel of $F(\varphi^{n})$, since $\pi^{n+1}d_N^{n}F(\varphi^{n})=\pi^{n+1}\varphi^{n+1}d_M^n=0$. Thus, $\mathfrak{L}^*_F(\mathcal{C})$ is also closed by cokernels. Finally, following similar lines of reasoning ``by components'' in the two previous proofs, it is easy to conclude that for every morphism $\varphi:M\longrightarrow N$ in $\mathfrak{L}^*_F(\mathcal{C})$ there exists a sequence $K\stackrel{\iota}{\longrightarrow} M\stackrel{i}{\longrightarrow} I\stackrel{j}{\longrightarrow} N\stackrel{\pi}{\longrightarrow} C$ satisfying: $ji=\varphi$; $(K,\iota)=\ker(\varphi)$ and $(C,\pi)=\textup{coker}(\varphi)$; $(I,i)=\textup{coker}(\iota)$ and $(I,j)=\ker(\pi)$, which completes the proof of item $(iii)$.

\item The proof is analogous to the proof of item $(ii)$, relying on the cokernel construction for every morphism $\varphi:M\longrightarrow N$ in $\mathfrak{R}^*_F(\mathcal{C})$. That is, $\left(\textup{coker}(\varphi),\pi\right)$, where $\pi=(\pi^n)_{n\in\mathbb{Z}}$ is defined from of the canonical projections $\pi^n:N^n\longrightarrow \textup{coker}(\varphi^n)$ and $\textup{coker}(\varphi)=\left(\textup{coker}(\varphi^n),d_{\textup{coker}(\varphi)}^n\right)_{n\in\mathbb{Z}}$, where $d_{\textup{coker}(\varphi)}^n:\textup{coker}(\varphi^n)\longrightarrow F(\textup{coker}(\varphi^{n+1}))$ the unique morphism with the property  $F(\pi^{n+1})d_N^{n}=d_{\textup{coker}(\varphi)}^n\pi^{n}$. Indeed, since $F(\pi^{n+1})d_N^{n}\varphi^{n}=F(\pi^{n+1})F(\varphi^{n+1})d_M^n=F(\pi^{n+1}\varphi^{n+1})d_M^n=0$, $d_{\textup{coker}(\varphi)}^n$ is obtained by universal property of the cokernels.

\item The proof is omitted due to its similarity to the item $(iii)$.

\item The claim  $\varphi+\lambda\psi\in \textup{Hom}_{\mathfrak{L}^*_F(\mathcal{C})}(M,N)$ for every $\lambda\in \Bbbk$ and morphisms $\varphi=(\varphi^n)_{n\in\mathbb{Z}},\,\psi=(\psi^n)_{n\in\mathbb{Z}}\in \textup{Hom}_{\mathfrak{L}^*_F(\mathcal{C})}(M,N)$, is a direct consequence of the following commutative diagrams:
\begin{figure}[!htb]
\begin{center}
\begin{minipage}{0.98\textwidth}
\xymatrix{
& F(M^n) \ar[r]^-{d_M^n}\ar[d]_-{F(\varphi^n)}& M^{n+1} \ar[d]^-{\varphi^{n+1}} & \\
& F(N^{n})\ar[r]_-{d_N^n}& N^{n+1}& 
}
\end{minipage}
\begin{minipage}{0.98\textwidth}
\xymatrix{
& F(M^n) \ar[r]^-{d_M^n}\ar[d]_-{\lambda F(\psi^n)}& M_{n+1} \ar[d]^-{\lambda\psi^{n+1}} & \\
& F(N^{n})\ar[r]_-{d_N^n}& N^{n+1}& 
}
\end{minipage}
\end{center}
\end{figure}
\begin{equation*}
\xymatrix{
& F(M^n) \ar[r]^-{d_M^n}\ar[d]_-{F(\varphi^n+\lambda\psi^n)}& M^{n+1} \ar[d]^-{\varphi^{n+1}+\lambda\psi^{n+1}} & \\
& F(N^{n})\ar[r]_-{d_N^n}& N^{n+1}& 
}
\end{equation*}
\end{enumerate}

\end{proof}

%----------------------------------------%

\begin{corollary}
\label{cor:seq}
Let $\mathcal{C}$ be an abelian category endowed with an endofunctor $F:\mathcal{C}\longrightarrow\mathcal{C}$, which preserves cokernels (resp. kernels). Then, the sequence
	$
	0\longrightarrow M\stackrel{\varphi}{\longrightarrow}N
\stackrel{\psi}{\longrightarrow} P\longrightarrow0
	$
is exact in $\mathfrak{L}^*_F(\mathcal{C})$ (resp. in $\mathfrak{R}^*_F(\mathcal{C})$) if and only if
	$
	0\longrightarrow M^n\stackrel{\varphi^n}{\longrightarrow} N^n\stackrel{\psi^n}{\longrightarrow} P^n\longrightarrow0
	$
are exact in $\mathcal{C}$, for all $n\in\mathbb{Z}$.
\end{corollary}
\begin{proof}
The equivalence follows as a direct consequence of Theorem~\ref{thm:abelian}.
\end{proof}

%----------------------------------------------%

\begin{remark}\rm
A natural consequence of Corollary \ref{cor:seq} is that, within Grothendieck categories or complete and cocomplete abelian categories, the exactness of sequences in categories such as $\mathfrak{L}^*_F(\mathcal{C})$ (resp.  $\mathfrak{R}^*_F(\mathcal{C})$) is linked to exactness in the underlying category $\mathcal{C}$, provided that $F$ satisfies the additional condition of preserving co-kernels (resp. kernels).
\end{remark}

To improve readability, Proposition~\ref{prop:funct}, Corollaries~\ref{cor:Fext} and~\ref{cor:Equiv} will be presented in the context of $\mathfrak{L}^*_F(\mathcal{C})$. However, these results can also be expressed in terms of $\mathfrak{R}^*_F(\mathcal{C})$, with the necessary adjustments.

\begin{proposition}
\label{prop:funct}
Let $E:\mathcal{C}\longrightarrow\mathcal{D}$ be a functor. If $F\in\textup{End}(\mathcal{C})$ and $G\in \textup{End}(\mathcal{D})$ are endofunctors for which there exists a natural transformation $\mu:G\circ E\longrightarrow E\circ F$, then this transformation and the functor $E$ induce an additive covariant functor 
    $$
    \begin{array}{rccl}
    \widehat{E}:&\mathfrak{L}^*_F(\mathcal{C})&\longrightarrow & \mathfrak{L}^*_G(\mathcal{D})\\
    & M &\longmapsto & \widehat{E}(M):=\left(E(M^n), d_{\widehat{E}(M)}^n  \right)_{n\in\mathbb{Z}}\\
    & \varphi &\longmapsto & \widehat{E}(\varphi):=\left(E(\varphi^n)\right)_{n\in\mathbb{Z}}
    \end{array}
    $$
where $d_{\widehat{E}(M)}^n$ corresponds to the composition $E(d_M^n)\circ\mu_{M^n}: G(E(M^n))\longrightarrow E(M^{n+1})$.
\end{proposition}

%-------------------------------------------%

\begin{proof}
We need to show that:
\begin{itemize}
\item $d_{\widehat{E}(M)}^{n+1}\circ G\big(d_{\widehat{E}(M)}^n\big)=0$ for all $n\in\mathbb{Z}$,
\item For each morphism $E(\varphi^n):E(M^n)\longrightarrow E(N^n)$ satisfies $E(\varphi^{n+1})\circ d_{\widehat{E}(M)}^n=d_{\widehat{E}(N)}^n\circ G(E(\varphi^n))$, for all morphism  $\varphi:M\longrightarrow N$ in $\mathfrak{L}^*_F(\mathcal{C})$.
\end{itemize}
 The first claim it follows from the following sequence of equalities:
\begin{eqnarray*}
d_{\widehat{E}(M)}^{n+1}\circ G\big(d_{\widehat{E}(M)}^n\big) &=& E\big(d_M^{n+1}\big)\circ \mu_{M^{n+1}}\circ G(E(d_M^n))\circ G(\mu_{M^n})\\
&=& E\big(d_M^{n+1}\big)\circ \big(\mu_{M^{n+1}}\circ G(E(d_M^n))\big)\circ G(\mu_{M^n})\\
&=& E(d_M^{n+1})\circ \big(E(F(d_M^n))\circ \mu_{F(M^n)}\big)\circ G(\mu_{M^n})\\
&=& E\big(d_M^{n+1}\circ F(d_M^n)\big)\circ \mu_{F(M^n)}\circ G(\mu_{M^n})\\
&=& E(0)\circ \mu_{F(M^n)}\circ G(\mu_{M^n})\\
&=& 0
\end{eqnarray*}
The second claim holds from the following sequence of equalities:
\begin{eqnarray*}
E\big(\varphi^{n+1}\big)\circ d_{\widehat{E}(M)}^n &=& E\big(\varphi^{n+1}\big)\circ E(d_M^n)\circ\mu_{M^n}\\
&=& E\big(\varphi^{n+1}\circ d_M^n\big)\circ\mu_{M^n}\\
&=& E\big(d_N^n\circ F(\varphi^{n})\big)\circ\mu_{M^n}\\
&=& E(d_N^n)\circ E(F(\varphi^n))\circ \mu_{M^n}\\
&=& E(d_N^n)\circ \mu_{N^n}\circ G(E(\varphi^n))\\
&=& d_{\widehat{E}(N)}^n\circ G(E(\varphi^n)).
\end{eqnarray*}
%Clearly, by definition, we obtain that $\widehat{E}$ is a covariant functor, which completes the proof.
\end{proof}

%-------------------------------------------%

\begin{corollary}\label{cor:Fext}
Every endofunctor $F:\mathcal{C}\longrightarrow\mathcal{C}$ can be extended to a covariant endofunctor     $$
    \begin{array}{rccl}
    \widehat{F}:&\mathfrak{L}^*_F(\mathcal{C})&\longrightarrow & \mathfrak{L}^*_F(\mathcal{C})\\
    & M &\longmapsto & \widehat{F}(M):=\left(F(M^n), F(d_M^n)\right)_{n\in\mathbb{Z}}\\
    & \varphi &\longmapsto & \widehat{F}(\varphi):=\left(F(\varphi^n)\right)_{n\in\mathbb{Z}}
    \end{array}
    $$
\end{corollary}

%-----------------------------------------%
\begin{proof}
It is sufficient to take $E=F=G$, $\mathcal{C}=\mathcal{D}$ and by the natural isomorphism $\mu$ the identity in Proposition \ref{prop:funct}.

\end{proof}

%-----------------------------------------%

\begin{corollary}
\label{cor:Equiv}
Let $E:\mathcal{C}\longrightarrow\mathcal{D}$ be an equivalence of categories, and let $F\in\textup{End}(\mathcal{C})$ be an endofunctor. Then, there exists a unique, up to isomorphism, additive covariant endofunctor $G\in \textup{End}(\mathcal{D})$  such that the induced covariant functor $\widehat{E}:\mathfrak{L}^*_F(\mathcal{C})\longrightarrow \mathfrak{L}^*_G(\mathcal{D})$ is an equivalence of categories.
\end{corollary}

%-----------------------------------------%

\begin{proof}
Consider the endofunctor $G:= E\circ F\circ H\in \textup{End}(\mathcal{D})$, where $H:\mathcal{D}\longrightarrow\mathcal{C}$ is a pseudo-inverse of the equivalence $E$. Since $G\circ E\simeq E\circ F$ and $F\circ H\simeq H\circ G$, we obtain that there exist natural isomorphisms $\mu:G\circ E\longrightarrow E\circ F$ and $\beta:F\circ H\longrightarrow H\circ G$. From Proposition \ref{prop:funct} we obtain two covariant additive functors 
    $$
    \begin{array}{cclcccccccl}
    \widehat{E}: \mathfrak{L}^*_F(\mathcal{C})&\longrightarrow & \mathfrak{L}^*_G(\mathcal{D})&&&&&&\widehat{H}:\mathfrak{L}^*_G(\mathcal{D})&\longrightarrow & \mathfrak{L}^*_F(\mathcal{C})\\
    \ \ \ M &\longmapsto & \left(E(M^n), d^n_{\widehat{E}(M)}\right)_{n\in\mathbb{Z}}&&&&&& \ \ \ N &\longmapsto & \left(H(N^n), d^n_{\widehat{H}(N)}\right)_{n\in\mathbb{Z}}\\
    \ \ \ \varphi &\longmapsto & \left(E(\varphi^n)\right)_{n\in\mathbb{Z}}&&&&&& \ \ \ \psi &\longmapsto & \left(H(\psi^n)\right)_{n\in\mathbb{Z}}
    \end{array}
    $$
It is easy to check that $\widehat{H}\circ\widehat{E}\simeq \textup{Id}_{\mathfrak{L}^*_F(\mathcal{C})}$ and $\widehat{E}\circ\widehat{H}\simeq \textup{Id}_{\mathfrak{L}^*_G(\mathcal{D})}$, which completes the proof.

\end{proof}

%-----------------------------------%

\begin{example}\rm
Let $f:B\longrightarrow A$ be an isomorphism of $\Bbbk$-algebras, with $g:A\longrightarrow B$ its inverse. The morphism $f$ induces a natural covariant $\Bbbk$-linear functor $E:A\textup{-Mod}\longrightarrow B\text{-Mod}$ defined over objects $M\longmapsto {}_{f}M$, where ${}_{f}M$ indicates the $B$-module structure induced by $f$ on $M$. Likewise, $g$ induces a $\Bbbk$-linear functor $H:B\text{-Mod}\longrightarrow A\text{-Mod}$ which is the inverse of $F_f$, that is, $E\circ H=\textup{Id}_{B\text{-Mod}}$ and $H\circ E=\textup{Id}_{A\text{-Mod}}$. In conclusion, $E$ is an isomorphism, {\it a fortiori}, an equivalence. More details in \cite[Lemma 6.3]{SY}.

Now, we consider for any $A$-bimodule $D$ the endofunctor $F:A\textup{-Mod}\longrightarrow A\textup{-Mod}$, defined by $M\longmapsto D\otimes M$. Following the ideas in Corollary \ref{cor:Equiv}, we obtain an endofunctor $G:B\textup{-Mod}\longrightarrow B\textup{-Mod}$ defined by $N\longmapsto {}_{f}D\otimes {}_{g}N$ and the equivalences $\widehat{E}$ and $\widehat{H}$ are determined by
    $$
    \begin{array}{cclcccccccl}
    \widehat{E}: \mathfrak{L}^*_F(A\textup{-Mod})&\longrightarrow & \mathfrak{L}^*_G(B\textup{-Mod}) &&&&&&\widehat{H}:\mathfrak{L}^*_G(B\textup{-Mod})&\longrightarrow & \mathfrak{L}^*_F(A\textup{-Mod})\\
    \ \ \ M &\longmapsto & \left({}_{f}M^n, d^n_{{}_{f}M}\right)_{n\in\mathbb{Z}}&&&&&& \ \ \ N &\longmapsto & \left({}_{g}N^n, d^n_{{}_{g}N})\right)_{n\in\mathbb{Z}}\\
    \ \ \ \varphi &\longmapsto & \left(\varphi^n\right)_{n\in\mathbb{Z}}&&&&&& \ \ \ \psi &\longmapsto & \left(\psi^n\right)_{n\in\mathbb{Z}}
    \end{array}
    $$
since, in this case, $\mu$ and $\beta$ are the identities. For sake of completeness, observe that any $M=(M^n,d^n_{M})\in \mathfrak{L}^*_F(A\textup{-Mod})$ which correspond a $M^n$ $A$-module (not necessarily finitely generated) and $A$-module homomorphism $d^n_{M}:D\otimes M^n\longrightarrow M^{n+1}$ are mapped, respectively, to ${}_{f}M^n$ and $d^n_{{}_{f}M}:{}_{f}D\otimes M^n\longrightarrow {}_{f}M^{n+1}$. Here, ${}_{f}D\otimes M^n$ is the result of applying ${}_{f}D\otimes M^n={}_{f}D\otimes {}_{gf}M^n$ and using that $g:A\longrightarrow B$ is the inverse of $f$.
\end{example}

% \begin{remark}\rm
% \textcolor{red}{Deberá ir otro comentario que debe estar relacionado a la nueva proposicion sobre las equivalencias y de que el functor no refleje isomorfismos.}
% This example serves also to illustrate Proposition~\ref{prop:catgroup}, which explores the types of equivalences that arise in a more general setting.
% In this broader context, a new perspective emerges, centered around novel structures, as the monoidal and group categorical structures. Further details can be found in Subsection \ref{subsec:5}.
% \end{remark}

%***********************************************************%	
			% Limits on $L_F(\mathcal{C})$ %
%***********************************************************%
\subsection{(Co-)Completeness}
Let $\mathcal{C}$ be a complete (resp. cocomplete) category, and $F:\mathcal{C}\longrightarrow\mathcal{C}$ be an endofunctor. The functor $F$ is called \emph{continuos} (resp. \emph{cocontinuos}) if it preserves (small) limits (resp. colimits). The main aim of this subsection is to guarantee and prove the conditions under which the category $\mathfrak{L}^*_F(\mathcal{C})$ is complete (resp. cocomplete). To this end, we will use results from Eilenberg-Moore categories to guarantee that these properties hold under expected conditions on $F$. An interesting consequence in our context, compared to the general Eilenberg-Moore theory, is that our limits and colimits can be made explicit. This includes products, coproducts, equalizers, and coequalizers, as shown below.
%----------------------------------------%

\begin{theorem}
\label{thm:prod and eq}
Let $\mathcal{C}$ be an additive category and $F\in\textup{End}(\mathcal{C})$ an additive endofunctor. Then, 
\begin{enumerate}[(i)]
\item If $\mathcal{C}$ is a complete category, then $\mathfrak{L}^*_F(\mathcal{C})$ is closed under limits. Specifically, given an inverse system $\left\{M_i,\varphi_{ij}\right\}_{i,j\in I}$ in $\mathfrak{L}^*_F(\mathcal{C})$ indexed by a poset $I$, its limit is the object 
$$
\lim\limits_{\stackrel{\longleftarrow}{i\in I}} M_i=\bigg(\lim\limits_{\stackrel{\longleftarrow}{i\in I}}M_{i}^{n},\lim\limits_{\stackrel{\longleftarrow}{i\in I}} d_{M_{i}}^n\bigg)_{n\in\mathbb{Z}}
$$ 
together with morphisms $\alpha_{i}:=\left(\alpha_{i}^{n}\right)_{n\in\mathbb{Z}}:  \lim\limits_{\stackrel{\longleftarrow}{i\in I}} M_i\longrightarrow M_i$ in $\mathfrak{L}^*_F(\mathcal{C})$, where the $n$-th component $\lim\limits_{\stackrel{\longleftarrow}{i\in I}}M_{i}^{n}$ with the morphisms $\alpha_{i}^{n}:\lim\limits_{\stackrel{\longleftarrow}{i\in I}} M_{i}^{n}\longrightarrow M_{i}^{n}$ is the limit in $\mathcal{C}$ of the inverse system $\left\{M_i^n,\varphi_{ij}^n\right\}_{i,j\in I}$. Additionally, the differential  
$$
\lim\limits_{\stackrel{\longleftarrow}{i\in I}} d_{M_{i}}^n:F\bigg(\lim\limits_{\stackrel{\longleftarrow}{i\in I}}M_{i}^{n}\bigg)\longrightarrow \lim\limits_{\stackrel{\longleftarrow}{i\in I}}M_{i}^{n+1}
$$ 
is the morphism induced by the universal property of limits on $\mathcal{C}$.

\item If $\mathcal{C}$ is a cocomplete category, and $F\in\textup{End}(\mathcal{C})$ is a cocontinuous endofunctor, then $\mathfrak{L}^*_F(\mathcal{C})$ is closed under colimits. Specifically, given a direct system $\left\{M_i,\varphi_{ij}\right\}_{i,j\in I}$ in $\mathfrak{L}^*_F(\mathcal{C})$ indexed by a poset $I$, the colimit is the object 
$$
\lim\limits_{\stackrel{\longrightarrow}{i\in I}} M_i=\bigg(\lim\limits_{\stackrel{\longrightarrow}{i\in I}}M_{i}^{n},\lim\limits_{\stackrel{\longrightarrow}{i\in I}} d_{M_{i}}^n\bigg)_{n\in\mathbb{Z}}
$$ 
endowed with morphisms $\alpha_{i}:=\left(\alpha_{i}^{n}\right)_{n\in\mathbb{Z}}: M_i\longrightarrow \lim\limits_{\stackrel{\longrightarrow}{i\in I}} M_i$ in $\mathfrak{L}^*_F(\mathcal{C})$. Here, the $n$-th component $\lim\limits_{\stackrel{\longrightarrow}{i\in I}}M_{i}^{n}$ together with the morphisms $\alpha_{i}^{n}:M_{i}^{n}\longrightarrow \lim\limits_{\stackrel{\longrightarrow}{i\in I}} M_{i}^{n}$ is the colimit of the direct system $\left\{M_i^n,\varphi_{ij}^n\right\}_{i,j\in I}$ in $\mathcal{C}$. Additionally, the differential
    $$
    \lim\limits_{\stackrel{\longrightarrow}{i\in I}} d_{M_{i}}^n:F\bigg(\lim\limits_{\stackrel{\longrightarrow}{i\in I}}M_{i}^{n}\bigg)\longrightarrow \lim\limits_{\stackrel{\longrightarrow}{i\in I}}M_{i}^{n+1}
    $$ 
is the morphism induced by the universal property of colimits on $\mathcal{C}$.
\end{enumerate}
In particular, if $\mathcal{C}$ is a complete and cocomplete category and $F\in\textup{End}(\mathcal{C})$ is a cocontinuos (resp. $G\in\textup{End}(\mathcal{C})$ continuos) endofunctor, then  $\mathfrak{L}^*_F(\mathcal{C})$ (resp. $\mathfrak{R}^*_G(\mathcal{C})$) is a complete and cocomplete category.
\end{theorem}

%-------------------------------------------%

\begin{proof} 
First of all, observe that the forgetful functor $U:\mathcal{C}^{T}\longrightarrow\mathcal{C}$ preserves graduations, where $\mathcal{C}^{T}$ is the Eilenberg-Moore category associated with the monad introduced in Theorem~\ref{thm:graduate}. Moreover, as established in the same theorem, the category $\mathfrak{L}^*_F(\mathcal{C})$ is isomorphic to the full subcategory of the Eilenberg-Moore category $\mathcal{C}^{T}$.

Consequently, item (i) follows directly from \cite[Theorem 4.3.1]{Bor94b}, while item (ii) follows from \cite[Theorem 4.3.2]{Bor94b} together with the fact that the endofunctor $T:\mathcal{C}\longrightarrow\mathcal{C}$ preserves colimits if and only if the endofunctor $F:\mathcal{C}\longrightarrow\mathcal{C}$ also preserves colimits.

Finally, by straightforward calculations, we can verify that the objects in the statement of the theorem satisfy the universal property of limits and colimits, respectively, which completes the proof. 
\end{proof}

%-------------------------------------------%

\begin{corollary}\label{thm:completness}
Let $\mathcal{C}$ be a complete and cocomplete category and $F:\mathcal{C}\longrightarrow\mathcal{C}$ an additive endofunctor. Then,
\begin{enumerate}[(i)]
\item $\mathfrak{L}^*_F(\mathcal{C})$ is closed under equalizers. Specifically, the equalizer of two morphisms $\varphi,\psi:M\longrightarrow N$ in $\mathfrak{L}^*_F(\mathcal{C})$ is the morphism $\kappa:K\longrightarrow M$ in $\mathfrak{L}^*_F(\mathcal{C})$ given by the equalizers $\kappa^n:K^n\longrightarrow M^n$ of the morphisms $\varphi^n,\psi^n:M^n\longrightarrow N^n$ in $\mathcal{C}$. And, the differential $F(K^{n})\longrightarrow K^{n+1}$ is the morphism induced by the universal property of equalizers on $\mathcal{C}$.

\item If $F$ preserves coequalizers, then $\mathfrak{L}^*_F(\mathcal{C})$ is closed under coequalizers. Specifically, the coequalizer of two morphisms $\varphi,\psi:M\longrightarrow N$ in $\mathfrak{L}^*_F(\mathcal{C})$ is the morphism $\lambda:N\longrightarrow L$ in $\mathfrak{L}^*_F(\mathcal{C})$ given by the coequalizers $\lambda^n:N^n\longrightarrow L^n$ of the morphisms $\varphi^n,\psi^n:M^n\longrightarrow N^n$ in $\mathcal{C}$. And, the differential $F(L^{n})\longrightarrow L^{n+1}$ is the morphism induced by the universal property of coequalizers on $\mathcal{C}$.

\item $\mathfrak{L}^*_F(\mathcal{C})$ is closed under products. Specifically, the product in $\mathfrak{L}^*_F(\mathcal{C})$ of a family $\left\{M_i=\left(M_{i}^{n},d^n_{M_{i}}\right)_{n\in\mathbb{Z}}\right\}_{i\in I}$ of objects indexed by the set $I$, is given by $\prod\limits_{i\in I}  M_i=\left(\prod\limits_{i\in I} M_{i}^{n},\prod\limits_{i\in I}d^n_{M_{i}}\right)_{n\in\mathbb{Z}}$.

\item If the endofuntor $F$ preserves coproducts, then $\mathfrak{L}^*_F(\mathcal{C})$ is closed under coproducts. Specifically, the coproduct in $\mathfrak{L}^*_F(\mathcal{C})$ of a family of objects $\left\{M_i=\left(M_{i}^{n},d^n_{M_{i}}\right)_{n\in\mathbb{Z}}\right\}_{i\in I}$ indexed by the set $I$, is given by $\coprod\limits_{i\in I}M_i=\left(\coprod\limits_{i\in I} M_{i}^{n},\coprod\limits_{i\in I}d^n_{M_{i}}\right)_{n\in\mathbb{Z}}$.
\end{enumerate}
\end{corollary}

%-------------------------------------------%
\begin{proof}
The proof follows the similar lines of reasoning in the proof of the existence of limits and colimits in Theorem \ref{thm:prod and eq}.
\end{proof}

%-------------------------------------------%

%-------------------------------------------%

\begin{corollary}\label{cor:adjconsq}
Let $\mathcal{C}$ be a complete and cocomplete abelian category. If $(F, G)$ is an adjoint pair of endofunctors over $\mathcal{C}$, then $\mathfrak{L}^*_F(\mathcal{C})$ and $\mathfrak{R}^*_G(\mathcal{C})$ are complete and cocomplete abelian categories.
\end{corollary}

%---------------------%
\begin{proof}
For an adjoint pair $(F, G)$, we recall that $F$ is cocontinuos and $G$ is continuos. By applying both Theorem \ref{thm:prod and eq} and Theorem \ref{thm:abelian}, we conclude that $\mathfrak{L}^*_F(\mathcal{C})$ and $\mathfrak{R}^*_G(\mathcal{C})$ are complete and cocomplete abelian categories. 
\end{proof}

%\newpage
%-------------------------------------------%

\begin{remark}\rm
\label{rem:pullback}
\begin{enumerate}[(i)]
\item Similar to Theorem \ref{thm:prod and eq}, equivalent results apply to $\mathfrak{R}^*_F(\mathcal{C})$ when $\mathcal{C}$ is a complete category and $F$ is a continuous functor, which guarantees the existence of limits in $\mathfrak{R}^*_F(\mathcal{C})$.

\item Let $\mathcal{C}$ be a complete category. An example of limit is the pullback of two morphisms $\varphi:M\longrightarrow N$ and $\psi:P\longrightarrow N$ in $\mathfrak{L}^*_F(\mathcal{C})$. In this case, the pullback is the pair $M\times_{N}P=\left(M^n\times_{N^n}P^n,d^n_{M\times_{N}P}\right)_{n\in\mathbb{Z}}$, where $M^n\times_{N^n}P^n$ is the pullback in $\mathcal{C}$ given by $\varphi^n:M^n\longrightarrow N^n$ and $\psi^n:P^n\longrightarrow N^n$, and $d^n_{M\times_{N}P}$ is obtained as in the proof of Proposition \ref{thm:prod and eq}(ii). Diagrammatically
	$$
	\xymatrix{
	& M^{n-1}\times_{N^{n-1}}P^{n-1} \ar[rr] \ar@{~>}[dl]_-{d^{n-1}_{M\times_{N}P}} \ar'[d][dd] & & P^{n-1} \ar[dd]^-{\psi^{n-1}}\ar@{~>}[dl]_-{d^{n-1}_P}\\
	M^{n}\times_{N^{n}}P^{n} \ar[rr] \ar[dd] & & P^{n} \ar[dd] \\
	& M^{n-1} \ar'[r]^{\varphi^{n-1}}[rr] \ar@{~>}[dl]_-{d^{n-1}_M} & & N^{n-1}\ar@{~>}[dl]^-{d^{n-1}_N}\\
    M^{n} \ar[rr]_-{\varphi^{n}} & & N^{n}
	}
	$$
In particular, an important case of this construction is when $M=0$ and $\varphi=0$. In this case, the pullback coincides with $\ker(\psi)=\left(\ker(\psi^n),d^n_{\ker(\psi)}\right)_{n\in\mathbb{Z}}$.
\end{enumerate}
\end{remark}

%-----------------------------------------------------------%

\vspace{0.3cm}

%-----------------------------------------------%

%In the realm either of Grothendieck categories or of complete and co-complete abelian categories, a natural question arises: how does exactness of sequences in categories like $\mathfrak{L}^*_F(\mathcal{C})$ (resp.  $\mathfrak{R}^*_F(\mathcal{C})$) relate to exactness in the underlying category $\mathcal{C}$? It is readily verifiable that exactness in these latter categories ensures exactness of corresponding sequences in $\mathcal{C}$ for each degree. However, the converse implication, that exactness in $\mathcal{C}$ for all degrees implies exactness in the category $\mathfrak{L}^*_F(\mathcal{C})$ (resp.  $\mathfrak{R}^*_F(\mathcal{C})$), requires additional conditions, as we will discover through applications the results established in this section in the following corollary.

%------------------------------------------%

%=============================================%
    %PROJECTIVE AND INJECTIVE OBJECTS%
%=============================================%

\section{Projective and injective objects}
\label{sec:proj_inj}
%---------------------------------------------%
In this section we will consider only abelian categories. Recall that, by Theorem \ref{thm:abelian}, $\mathfrak{L}_F^{\ast}(\mathcal{C})$ is an abelian category if ($\mathcal{C}$ is an abelian category and) $F$ is a right-exact endofunctor. Similarly, $\mathfrak{R}_G^{\ast}(\mathcal{C})$ is an abelian category if $G$ is a left-exact endofunctor.

From Example \ref{example:L_F(C)}(ii) we know that $\mathfrak{L}^b_{DA\otimes_A -}\left(A\text{-mod}\right)$ is the ca\-te\-go\-ry of $\widehat{A}$-mod. Likewise, the latter category has enough projectives, and the classes of injective and projective modules coincide, that is, it is a \textit{Frobenius category}. Motivated by these constructions, we will dedicate this section to determining projective and injective objects on $\mathfrak{L}^*_F(\mathcal{C})$ and $\mathfrak{R}^*_F(\mathcal{C})$.
%, when $\mathcal{C}$ is an abelian category. 

Let us start by introducing a useful object. For any morphism $\psi\in\textup{Hom}_{\mathcal{C}}(F(X),Y)$ (respectively, $\psi\in\textup{Hom}_{\mathcal{C}}(X,F(Y))$), we will denote by $\mathfrak{C}^k(\psi)$ the object in $\mathfrak{L}^*_F(\mathcal{C})$ (resp. in $\mathfrak{R}^*_F(\mathcal{C})$) concentrated in degrees $(k,k+1)$. That is, $X$ is the term in the $k$th position, $Y$ is the term in the $(k+1)$th position, all other terms are $0$, and $\psi$ in the morphism from $k$th position to $(k+1)$th position. Graphically, $\mathfrak{C}^k(\psi)$ is represented by
	$$
	\xymatrix{
	\mathfrak{C}^k(\psi):=\cdots\ar@{~>}[r] & 0\ar@{~>}[r] & X\ar@{~>}[r]^-{\psi} & Y\ar@{~>}[r] & 0\ar@{~>}[r] & \cdots}
	$$

From now on, a pair of adjoint functors (or an adjoint pair of functors) with $F$ left adjoint and $G$ right adjoint will be denoted by $(F,G)$. The mappings $\mathfrak{C}^k(1_{F(-)})$ and $\mathfrak{C}^k(\varepsilon_{-})$ naturally induce covariant functors from $\mathcal{C}$ to $\mathfrak{L}^{*}_F(\mathcal{C})$, where $\varepsilon:FG\longrightarrow \textup{Id}_{\mathcal{C}}$ is the counit of the adjoint pair $(F,G)$. Similarly, the mappings $\mathfrak{C}^k(1_{G(-)})$ and $\mathfrak{C}^k(\eta_{-})$ define covariant functors from $\mathcal{C}$ to $\mathfrak{R}^{*}_F(\mathcal{C})$, with $\eta:\textup{Id}_{\mathcal{C}}\longrightarrow GF$ the unit of the adjoint pair $(F,G)$.

%-------------------------------------------%

\begin{theorem}
\label{proj&inj}
Let $(F, G)$ be an adjoint pair of endofunctors over $\mathcal{C}$, where $\eta:\textup{Id}_{\mathcal{C}}\longrightarrow GF$ and  $\varepsilon:FG\longrightarrow \textup{Id}_{\mathcal{C}}$ are its unit and counit, respectively. Then, for each $k\in\mathbb{Z}$,
\begin{enumerate}[(i)]
\item $P$ is a projective object in $\mathcal{C}$ if and only if 
	$$
	\xymatrix{
	\mathfrak{C}^k\big(1_{F(P)}\big):=\cdots\ar@{~>}[r] & 0\ar@{~>}[r] & P\ar@{~>}[r]^-{1_{F(P)}} & F(P)\ar@{~>}[r] & 0\ar@{~>}[r] & \cdots
	}
	$$
   is a projective object in $\mathfrak{L}^*_F(\mathcal{C})$. Analogously, $P$ is a projective object in $\mathcal{C}$ if and only if 
	$$
	\xymatrix{
	\mathfrak{C}^k\big(\eta_P\big):=\cdots\ar@{~>}[r] & 0\ar@{~>}[r] & P\ar@{~>}[r]^-{\eta_P} & F(P)\ar@{~>}[r] & 0\ar@{~>}[r] & \cdots
	}
	$$
 is a projective object in $\mathfrak{R}^*_G(\mathcal{C})$. 

\item $I$ is an injective object in $\mathcal{C}$ if and only if 
	$$
	\xymatrix{
	\mathfrak{C}^k(\varepsilon_I):=\cdots\ar@{~>}[r] & 0\ar@{~>}[r] & G(I)\ar@{~>}[r]^-{\varepsilon_I} & I\ar@{~>}[r] & 0\ar@{~>}[r] & \cdots
	}
	$$
 is an injective object in $\mathfrak{L}^*_F(\mathcal{C})$. Analogously, $I$ is an injective object in $\mathcal{C}$ if and only if
	$$
	\xymatrix{
	\mathfrak{C}^k\big(1_{G(I)}\big):=\cdots\ar@{~>}[r] & 0\ar@{~>}[r] & G(I)\ar@{~>}[r]^-{1_{G(I)}} & I\ar@{~>}[r] & 0\ar@{~>}[r] & \cdots
	}
	$$
is an injective object in $\mathfrak{R}^*_G(\mathcal{C})$.

\end{enumerate}
\end{theorem}
%---------------------------------%

\begin{proof}
We will prove claims $(i)$ and $(ii)$ for the category $\mathfrak{L}^*_F(\mathcal{C})$. The proofs for $\mathfrak{R}^*_F(\mathcal{C})$ are completely analogous. For each $k\in\mathbb{Z}$, consider the functor $\pi_k:\mathfrak{L}^*_F(\mathcal{C})\longrightarrow\mathcal{C}$, which maps an $F$-cochain complex to its $k$th degree entry. It is easy to verify that $(\mathfrak{C}^k(1_{F(-)}),\pi_k)$ and $(\pi_k,\mathfrak{C}^k(\varepsilon_{-}))$ form adjoint pairs. This gives us the following natural bijections:
\begin{equation}\label{eq:natural}
\textup{Hom}_{\mathfrak{L}^*_F(\mathcal{C})}(\mathfrak{C}^k(1_{F(P)}),-)\cong\textup{Hom}_{\mathcal{C}}(P,\pi_k(-))
    \ \ \ \ \ \text{and}\ \ \ \ \ 
\textup{Hom}_{\mathfrak{L}^*_F(\mathcal{C})}(-,\mathfrak{C}^k(\varepsilon_{I}))\cong\textup{Hom}_{\mathcal{C}}(\pi_k(-),I)
\end{equation}
for objects $P$ and $I$ in $\mathcal{C}$. Using the first bijection, we obtain that $P$ is projective in $\mathcal{C}$ if and only if the functor $\textup{Hom}_{\mathcal{C}}(P,\pi_k(-))\cong\textup{Hom}_{\mathfrak{L}^*_F(\mathcal{C})}(\mathfrak{C}^k(1_{F(P)}),-)$ is exact. This is equivalent to saying that $\mathfrak{C}^k(1_{F(P)})$ is projective in $\mathfrak{L}^*_F(\mathcal{C})$. A similar conclusion follows for injective objects $I\in\mathcal{C}$ when using the second bijection in \ref{eq:natural}.
\end{proof}

%---------------------------------%

%------------------------------------%

%---------------------------------%
\vspace{,4cm}

For improved exposition, we will focus on $\mathfrak{L}^*_F(\mathcal{C})$. However, the results can be rewritten in terms of $\mathfrak{R}^*_F(\mathcal{C})$, with the necessary adjustments.

%---------------------------------%

\begin{definition}
\label{def:fundexact}
For each $n\in\mathbb{Z}$, the \emph{$n$th fundamental exact sequences} of an object $M\in\mathfrak{L}^*_F(\mathcal{C})$ are:
	$$
	\xymatrix@R=7pt{
	0\ar[r] & \textup{Im}\big(F\big(d^{n-1}_{M}\big)\big)\ar@{^{(}->}[r]^-{\iota^{n-1}}& \ker(d^{n}_{M})\ar@{->>}[r]& \ker\big(d^{n}_{M}\big)/\textup{Im}\big(F\big(d^{n-1}_{M}\big)\big)\ar[r] & 0\\
	0\ar[r] & \ker\big(d^{n}_{M}\big)\ar@{^{(}->}[r]^-{\widehat{i^n}}& F(M^{n})\ar[r]^-{\delta^n}& \textup{Im}\big(d^{n}_{M}\big)\ar[r] & 0
	}
	$$
where $\delta^n:F(M^n)\twoheadrightarrow\textup{Im}\big(d^n_{M}\big)$ is the unique epimorphism such that $d^n_M=i^n\delta^n$, with $i^n:\textup{Im}\big(d^n_{M}\big)\hookrightarrow M^{n+1}$ is the canonical inclusion. We will say that $M$ is \emph{split} if all its fundamental exact sequences are split in $\mathcal{C}$.
\end{definition}

%---------------------------------------------%

\begin{remark}\rm
\label{rm:fundexact}
Let $F:\mathcal{C}\longrightarrow\mathcal{C}$ be a right exact functor and $M\in \mathfrak{L}^*_F(\mathcal{C})$. From Definition \ref{def:fundexact}, we obtain
	$$
	\xymatrix@C=18pt@R=24pt{
	F^2(M^n)\ar[rr]^-{F(d_M^n)}\ar@{ ->>}[rd]_-{F(\delta^n)}& & F(M^{n+1})\ar[rr]^-{d^{n+1}_M}\ar@{ ->>}[rd]_-{\delta^{n+1}}& & M^{n+2}\\
	& F(\textup{Im}(d^n_M))=\textup{Im}(F(d^n_M))\ar[ru]_-{F(i^n)} &  & \textup{Im}(d^{n+1}_M)\ar@{^{(}->}[ru]_-{i^{n+1}}}
	$$
Since $\iota^n$ and $\widehat{i^{n+1}}$ are canonical inclusion, it follows that $F(d_M^n)=\widehat{i^{n+1}}\iota^nF(\delta^n)$. Consequently, $F(i^n)=\widehat{i^{n+1}}\iota^n$, because $F(\delta^n)$ is an epimorphism in $\mathcal{C}$ and, by universal property of kernel, the inclusion $\iota^n$ is the unique morphism in $\mathcal{C}$ with that property. Therefore, $F(i^n)$ is a monomorphism and the morphism in the left square in the diagram below, are all commutative:
	$$
	\xymatrix@C=18pt@R=24pt{
	& F^2(M^n)\ar[rr]^-{F(d^n_M)}\ar@{ ->>}[d]_-{F(\delta^n)}& & F(M^{n+1})\ar[rr]^-{d^{n+1}_M}& & M^{n+2}&\\
	0\ar[r]& F(\textup{Im}(d^n_M))=\textup{Im}(F(d^n_M))\ar[rru]_-{F(i^n)}\ar@{^{(}->}[rr]_-{\iota^{n}}& &  \ker\left(d^{n+1}_M\right)\ar@{^{(}->}[u]_-{\widehat{i^{n+1}}}\ar@{->>}[rr]&  & \ker\big(d^{n+1}_{M}\big)/\textup{Im}\big(F\big(d^{n}_{M}\big)\big)\ar[r] & 0
	}
	$$
\end{remark}

%------------------------------------------%
\vspace{.5cm}
In the following proposition we will use the extended endofunctor $\widehat{F}:\mathfrak{L}^*_F(\mathcal{C})\longrightarrow \mathfrak{L}^*_F(\mathcal{C})$ constructed in Corollary~\ref{cor:Fext}.

%------------------------------------------%

\begin{proposition}
\label{split}
Consider $F:\mathcal{C}\longrightarrow\mathcal{C}$ a right exact endofunctor. If $M\in \mathfrak{L}^*_F(\mathcal{C})$ is split, then we have an isomorphism:
	$$
	\widehat{F}(M)\cong\bigoplus_{k\in\mathbb{Z}}\mathfrak{C}^k(\iota^k).
	$$
\end{proposition}

%-----------------------------------------------------------%

\begin{proof}
Suppose that all fundamental exact sequences of $M$ are split (see Definition~\ref{def:fundexact}).
This implies that there exist morphisms $p^n:F(M^n)\longrightarrow\ker\left(d^n_{M}\right)$ and $q^n:\textup{Im}\left(d^n_{M}\right)\longrightarrow F(M^{n})$ in $\mathcal{C}$ such that
	$$
	\xymatrix{
	0\ar[r] & \ker\big(d^{n}_{M}\big)\ar@<2pt>[r]^{\widehat{i^n}}& F(M^{n})\ar@<2pt>[r]^{\delta^n}\ar@<2pt>[l]^{p^n}& \textup{Im}(d^{n}_{M})\ar[r]\ar@<2pt>[l]^-{q^n} & 0,
	}
	$$
which satisfying
 
	$$
	p^n\widehat{i^n}=1_{\ker\left(d^{n}_{M}\right)},\ \ \ \delta^nq^n=1_{\textup{Im}\left(d^{n}_{M}\right)},\ \ \ \delta^n\widehat{i^n}=0,\ \ \ p^nq^n=0\ \ \mbox{and}\ \ \widehat{i^n}p^n+q^n\delta^n=1_{F\left(M^{n}\right)}.
	$$
Now, since
	$$
	\xymatrix@R=10pt@C=11pt{
	\mathfrak{C}^{k-1}(\iota^{k-1})=\cdots\ar@{~>}[r] & 0\ar@{~>}[r] & \textup{Im}(d^{k-1}_M)\ar@{~>}[r]^-{\iota^{k-1}} & \ker(d^{k}_M)\ar@{~>}[r] & 0\ar@{~>}[r] & 0\ar@{~>}[r] & 0\ar@{~>}[r] & \cdots\\
	\ \ \ \ \ \ \mathfrak{C}^{k}(\iota^{k})=\cdots\ar@{~>}[r] & 0\ar@{~>}[r] & 0\ar@{~>}[r] & \textup{Im}(d^{k}_M)\ar@{~>}[r]^-{\iota^{k}} & \ker(d^{k+1}_M)\ar@{~>}[r] & 0\ar@{~>}[r] & 0\ar@{~>}[r] & \cdots\\
	\mathfrak{C}^{k+1}(\iota^{k+1})=\cdots\ar@{~>}[r] & 0\ar@{~>}[r] & 0\ar@{~>}[r] & 0\ar@{~>}[r] & \textup{Im}(d^{k+1}_M)\ar@{~>}[r]^-{\iota^{k+1}} & \ker(d^{k+2}_M)\ar@{~>}[r] & 0\ar@{~>}[r] & \cdots
	}
	$$
We obtain that $\bigoplus\limits_{k\in\mathbb{Z}}\mathfrak{C}^k(\iota^k)$ is described by
	$$
	\xymatrix@C=14pt{
	\cdots\ar@{~>}[r] & \ker(d^{k-1}_M)\oplus \textup{Im}(d^{k-1}_M)\ar@{~>}[rr]^-{\Delta^{k-1}} & & \ker(d^{k}_M)\oplus\textup{Im}(d^{k}_M)\ar@{~>}[rr]^-{\Delta^{k}} & & \ker(d^{k+1}_M)\oplus\textup{Im}(d^{k+1}_M)\ar@{~>}[r]^-{} & \cdots
	}
	$$
where $\Delta^{k}=\left[\begin{array}{cc}
        0 & \iota^{k} \\
        0 & 0
        \end{array}\right]$. Defining the collections of morphisms in $\mathcal{C}$
	$$
	\psi^k:=\left[\begin{array}{c}
        p^k \\
        \delta^k
        \end{array}\right]:F\big(M^k\big)\longrightarrow\ker(d^{k}_M)\oplus\textup{Im}\big(d^{k}_M\big)
	$$ 
and 
	$$
	\varphi^k=\left[\begin{array}{cc}
        \widehat{i^k} & q^k
        \end{array}\right]:\ker(d^{k}_M)\oplus\textup{Im}(d^{k}_M)\longrightarrow F\big(M^k\big)
	$$
we need to prove first that the morphisms
	$$
	\psi:=\left(\psi^n\right)_{n\in\mathbb{Z}}:\widehat{F}(M)\longrightarrow \bigoplus_{k\in\mathbb{Z}}\mathfrak{C}^k(\iota^k)\ \ \ \ \ \text{and }\ \ \ \ 	\varphi:=\left(\varphi^n\right)_{n\in\mathbb{Z}}:\bigoplus_{k\in\mathbb{Z}}\mathfrak{C}^k(\iota^k)\longrightarrow \widehat{F}(M)
	$$ 
belong to $\mathfrak{L}^*_F(\mathcal{C})$. Specifically, we need to prove that the following diagrams commutes
	 $$
	\xymatrix@C=11pt{
	F^2(M^n)\ar[r]^-{F(d^n_M)}\ar[d]_{F(\psi^n)} & F(M^{n+1})\ar[d]^{\psi^{n+1}} &  F\big(\ker(d^{n}_M)\oplus\textup{Im}(d^{n}_M)\big)\ar[r]^-{\Delta^n}\ar[d]_-{F(\varphi^n)} & \ker(d^{n+1}_M)\oplus\textup{Im}(d^{n+1}_M)\ar[d]^-{\varphi^{n+1}} \\
	F\big(\ker(d^{n}_M)\oplus\textup{Im}(d^{n}_M)\big)\ar[r]_-{\Delta^n}  & \ker(d^{n+1}_M)\oplus\textup{Im}(d^{n+1}_M) &  F^2(M^{n})\ar[r]_-{F(d^n_M)}  & F(M^{n+1})
	}
	$$
Indeed, note that in the left square we have that
	$$
	\Delta^nF(\psi^n)
	=\left[\begin{array}{cc}
        0 & \iota^{n} \\
        0 & 0
        \end{array}\right]\left[\begin{array}{c}
        F(p^n) \\
        F(\delta^n)
        \end{array}\right]
    =\left[\begin{array}{c}
        \iota^nF(\delta^n) \\
        0
        \end{array}\right]
	$$
	$$
	\psi^{n+1}F(d^n_M)
	=\left[\begin{array}{c}
        p^{n+1} \\
        \delta^{n+1}
        \end{array}\right]F(d^n_M)
    =\left[\begin{array}{c}
        p^{n+1}F(d^n_M) \\
        \delta^{n+1}F(d^n_M)
        \end{array}\right]
    =\left[\begin{array}{c}
        p^{n+1}F(i^n)F(\delta^n) \\
        \delta^{n+1}F(i^n)F(\delta^n)
        \end{array}\right]
	$$
Since $i^{n+1}\delta^{n+1}F(i^n)F(\delta^n)=d^{n+1}_MF(d^n_M)=0$ (see Definition~\ref{def:fundexact}) and $i^{n+1}:\textup{Im}\big(d^{n+1}_{M}\big)\longrightarrow F(M^{n+2})$ is a monomorphism, it follows that $\delta^{n+1}F(i^n)F(\delta^n)=0$. Thus,
	$$
	\psi^{n+1}F(d^n_M)
	=\left[\begin{array}{c}
        p^{n+1}F(i^n)F(\delta^n) \\
        0
        \end{array}\right]
	$$
In conclusion, since $F(i^n)=\widehat{i^{n+1}}\iota^n$ (see Remark~\ref{def:fundexact}) and $p^{n+1}\widehat{i^{n+1}}=1_{\ker(d^{n+1}_{M})}$, we obtain  $p^{n+1}F(i^n)=p^{n+1}\widehat{i^{n+1}}\iota^n=\iota_n$ and, consequently, $\Delta^nF(\psi^n)=\psi^{n+1}F(d^n_M)$, that is, the commutativity in the left square.

For the commutativity in the right square, note that
	$$
	\varphi^{n+1}\Delta^n
	=\left[\begin{array}{cc}
        \widehat{i^{n+1}} & q^{n+1}
        \end{array}\right]\left[\begin{array}{cc}
        0 & \iota^{n} \\
        0 & 0
        \end{array}\right]
    =\left[\begin{array}{cc}
        0 & \widehat{i^{n+1}}\iota^n
        \end{array}\right]
     =\left[\begin{array}{cc}
      0 & F(i^n)
        \end{array}\right]
	$$
	\begin{align*}
	F(d^n_M)F(\varphi^{n})
	&=F(d^n_M)\left[\begin{array}{cc}
        F(\widehat{i^{n}}) & F(q^{n})
        \end{array}\right]
    =\left[\begin{array}{cc}
      F(d^n_M)F(\widehat{i^{n}}) & F(d^n_M)F(q^{n})
        \end{array}\right]\\
	&=\left[\begin{array}{cc}
      0 & F(i^n)F(\delta^n)F(q^{n})
        \end{array}\right]
     =\left[\begin{array}{cc}
      0 & F(i^n)
        \end{array}\right]
	\end{align*}
	Hence $\varphi^{n+1}\Delta^n=F(d^n_M)F(\varphi^{n})$, which prove the commutativity of the aforementioned square. 
	
Finally, for all $n\in\mathbb{Z}$, it is only remaining to note that we have
	$$
	\varphi^n\psi^n=\left[\begin{array}{cc}
        \widehat{i^n} & q^n
        \end{array}\right]\left[\begin{array}{c}
        p^n \\
        \delta^n
        \end{array}\right]=
        \widehat{i^n}p^n+q^n\delta^n=1_{F(M^n)}
	$$
	and
	\begin{align*}
	\psi^n\varphi^n
	&=\left[\begin{array}{c}
        p^n \\
        \delta^n
        \end{array}\right]\left[\begin{array}{cc}
        \widehat{i^n} & q^n
        \end{array}\right]
	=\left[\begin{array}{cc}
        p^n\widehat{i^n} & p^nq^n\\
        \delta^n\widehat{i^n} & \delta^nq^n
        \end{array}\right]=\left[\begin{array}{cc}
        1_{\ker(d^n_M)} & 0\\
        0 & 1_{\textup{Im}(d^n_M)}
        \end{array}\right]
    =1_{\ker(d^n_M)\oplus\textup{Im}(d^n_M)}
	\end{align*}
which allows us to obtain that $\widehat{F}(M)\cong\bigoplus\limits_{k\in\mathbb{Z}}\mathfrak{C}^k(\iota^k)$ in $\mathfrak{L}^*_{F}(\mathcal{C})$.

\end{proof}

%---------------------------------------------%
\begin{corollary}\label{cor:projtype}
Consider $F,G:\mathcal{C}\longrightarrow\mathcal{C}$ two endofunctors such that $(F,G)$ is an adjoint pair. If $M\in\mathfrak{L}^*_F(\mathcal{C})$ is split such that each one of the terms is projective in $\mathcal{C}$ and $\textup{Im}\big(F\big(d^{n-1}_{M}\big)\big)=\ker\big(d^{n}_{M}\big)$ for all $n\in\mathbb{Z}$, then $\widehat{F}(M)$ is projective in $\mathfrak{L}^*_F(\mathcal{C})$.
\end{corollary}

%---------------------------------------------%

\begin{proof}
First of all, note that $1_{\ker\left(d^{n}_{M}\right)}=\iota^{n-1}:\textup{Im}\big(F\big(d^{n-1}_{M}\big)\big)\longrightarrow \ker\big(d^{n}_{M}\big)$ and thus
	$$
	\xymatrix{
	\mathfrak{C}^{n-1}(\iota^{n-1})=\cdots\ar[r] & 0\ar[r] & \ker\big(d^{n}_{M}\big)\ar[r]^-{1_{\ker\left(d^{n}_{M}\right)}} & \ker\big(d^{n}_{M}\big)\ar[r] & 0\ar[r] & \cdots}
	$$
Since $M\in\mathfrak{L}^*_F(\mathcal{C})$ is split, the following exact sequence 
	$$
	\xymatrix{
	0\ar[r] & \ker\big(d^{n}_{M}\big)\ar@{^{(}->}[r]^-{\widehat{i^n}}& F\big(M^{n}\big)\ar[r]^-{\delta^n}& \textup{Im}\big(d^{n}_{M}\big)\ar[r] & 0
	}
	$$
is split. Therefore, we can conclude that the complex $\mathfrak{C}^{n-1}(\iota^{n-1})$
is a direct summand of the complex	
    $$
	\xymatrix{
	\mathfrak{C}^{n-1}(1_{F(M^n)})=\cdots\ar[r] & 0\ar[r] & F(M^n)\ar[r]^-{1_{F(M^n)}} & F(M^n)\ar[r] & 0\ar[r] & \cdots}
	$$
in $\mathfrak{L}^*_{\text{Id}_{\mathcal{C}}}(\mathcal{C})$, because the exact sequence below
	$$
	\xymatrix{
	& 0\ar[d] & 0\ar[d] & 0\ar[d] & 0\ar[d] & \\
	\mathfrak{C}^{n-1}(\iota^{n-1})=\cdots\ar[r] & 0\ar[r]\ar[d] & \ker\left(d^{n}_{M}\right)\ar[r]^-{1_{\ker\left(d^{n}_{M}\right)}}\ar[d]^-{\widehat{i^n}} & \ker\left(d^{n}_{M}\right)\ar[r]\ar[d]^-{\widehat{i^n}} & 0\ar[r]\ar[d] & \cdots\\
	\mathfrak{C}^{n-1}(1_{F(M^n)})=\cdots\ar[r] & 0\ar[r]\ar[d] & F(M^n)\ar[r]^-{1_{F(M^n)}}\ar[d]^-{\delta^n} & F(M^n)\ar[r]\ar[d]^-{\delta^n} & 0\ar[r]\ar[d] & \cdots\\
	\cdots\ar[r] & 0\ar[r]\ar[d] &  \textup{Im}\left(d^{n}_{M}\right)\ar[r]^-{1_{ \textup{Im}\left(d^{n}_{M}\right)}}\ar[d] &  \textup{Im}\left(d^{n}_{M}\right)\ar[r]\ar[d] & 0\ar[r]\ar[d] & \cdots\\
	 & 0 & 0 & 0 & 0 & }
	$$
is split in $\mathfrak{L}^*_{\text{Id}_{\mathcal{C}}}(\mathcal{C})$.	Hence, by Theorem \ref{proj&inj}, $\mathfrak{C}^{n-1}(\iota_{n-1})$ is projective in $\mathfrak{L}^*_F(\mathcal{C})$. Finally, by Theorem \ref{split}, we conclude that $\widehat{F}(M)\cong \bigoplus\limits_{k\in\mathbb{Z}}\mathfrak{C}^k(\iota^k)$ is projective in $\mathfrak{L}^*_F(\mathcal{C})$, which completes the proof.

\end{proof}

%-----------------------------------------------------------%
\begin{remark}\rm
The converse of Corollary \ref{cor:projtype} does not hold generally. Consider $F=DA\otimes_A-: A\textup{-mod}\longrightarrow A\textup{-mod}$, where $A$ is a finite dimensional $\Bbbk$-algebra. Giraldo (\cite[Proposition 6]{Gir18}) showed that all the indecomposable projectives in $\mathfrak{L}^b_{DA\otimes_A-}(A\textup{-mod})$ has the form
    $$
	\xymatrix{
	\cdots\ar@{~>}[r] & 0\ar@{~>}[r] & A\varepsilon \ar@{~>}[r] & D(\varepsilon A)\ar@{~>}[r] & 0\ar@{~>}[r] & \cdots
	}
	$$
where $\varepsilon$ is an idempotent of $A$, implying $D(\varepsilon A)$ is an injective $A$-module. However, for the identity functor $F=\textup{Id}_{\mathcal{C}}$, where $\mathcal{C}$ is any abelian category, the converse indeed holds (cf. \cite[Exercise 2.2.1, p. 34]{Wei94}).

\end{remark}

%----------------------------------%

%===========================================================%

%=======================================%

\section{The (monoidal) category  \texorpdfstring{$\textup{End}(\mathcal{C})$}{eq:End(C)}, and the 2-categories \texorpdfstring{$\mathfrak{L}^*(\mathcal{C})$}{eq:LF(C)} and \texorpdfstring{$\mathfrak{R}^*(\mathcal{C})$}{eq:RF(C)}}
\label{sec:LF(C)vsEnd(C)}
In this section, we introduce two full sub-2-categories, $\mathfrak{L}^*(\mathcal{C})$ and $\mathfrak{R}^*(\mathcal{C})$, within the $2$-category of categories. Their objects are the categories $\mathfrak{L}^*_F(\mathcal{C})$ and $\mathfrak{R}^*_F(\mathcal{C})$, respectively, for each (additive) endofunctor $F\in\textup{End}(\mathcal{C})$ over an additive category $\mathcal{C}$. This functorial relationship is described in Theorem \ref{thm:LR}, and the subsequent results will explore some properties and connections between these functors and categories.

\subsection{The functors \texorpdfstring{$\mathfrak{L}^*$}{eq:L} and \texorpdfstring{$\mathfrak{R}^*$}{eq:R}}\label{subsec:natfunct}
Let us start by showing that any natural transformation $\Psi:F\longrightarrow G$ between endofunctors $F, G\in\textup{End}(\mathcal{C})$, induces the following covariant functor
    $$
    \begin{array}{rcl}
    \mathfrak{L}^*_{\Psi}(\mathcal{C}):\mathfrak{L}^*_G(\mathcal{C})&\longrightarrow & \mathfrak{L}^*_F(\mathcal{C})\\
   M \ \ \  &\longmapsto & \mathfrak{L}^*_{\Psi}(\mathcal{C})(M):=\left(M^n,d^n_M\circ\Psi_{M^n}\right)_{n\in\mathbb{Z}}\\
   \varphi \ \ \ \ &\longmapsto & \mathfrak{L}^*_{\Psi}(\mathcal{C})(\varphi):=\varphi
    \end{array}
    $$

We first establish the well-definedness of $\mathfrak{L}^*_{\Psi}(\mathcal{C})$. To achieve this, we show that for any $M=\left(M^n,d^n_M\right)_{n\in\mathbb{Z}}\in\mathfrak{L}^*_G(\mathcal{C})$, the corresponding object $\left(M^n,d^n_M\circ\Psi_{M^n}\right)_{n\in\mathbb{Z}}$ is in $\mathfrak{L}^*_F(\mathcal{C})$. In fact, we obtain the following sequence of equalities: 
    \begin{align*}
    d^{n+1}_M\circ\Psi_{M^{n+1}}\circ F\big(d^n_M\circ\Psi_{M^n}\big) 
    & = d^{n+1}_M\circ\left(\Psi_{M^{n+1}}\circ F\big(d^n_M\big)\right)\circ F\big(\Psi_{M^n}\big)\\
    & = d^{n+1}_M\circ\left(G\big(d^n_M\big)\circ \Psi_{G(M^{n})}\right)\circ F\big(\Psi_{M^n}\big)\\
    & = 0.
    \end{align*}
The final equality comes from the hypothesis that $M\in\mathfrak{L}^*_G(\mathcal{C})$. The second equality holds due to the commutativity of the following diagram:
    $$
    \xymatrix{
    F(G(M^n)) \ar[r]^{\Psi_{G(M^n)}}\ar[d]_{F(d_M^n)} & G(G(M^{n})) \ar[d]^{G(d_M^n)} \\
    F(M^{n+1})\ar[r]_{\Psi_{M^{n+1}}} & G(M^{n+1}), 
    }
    $$
guaranteed by the naturality of $\Psi:F\longrightarrow G$. Now, we need to prove that for each morphism $\varphi\in\textup{Hom}_{\mathfrak{L}^*_{G}(\mathcal{C})}\left(M,N\right)$, the corresponding morphism satisfies that $\varphi\in\textup{Hom}_{\mathfrak{L}^*_{F}(\mathcal{C})}\left(\mathfrak{L}^*_{\Psi}(\mathcal{C})(M),\mathfrak{L}^*_{\Psi}(\mathcal{C})(N)\right)$. Indeed, this claim is established by commutativity of the following diagram:
    $$
    \xymatrix{
    F(M^n)\ar[r]^-{\Psi_{M^n}}\ar[d]_{F(\varphi^n)} & G(M^n)\ar[r]^-{d^n_M} \ar[d]_{G(\varphi^n)} & M^{n+1} \ar[d]^{\varphi^{n+1}} \\
    F(N^{n}) \ar[r]_-{\Psi_{N^n}} & G(N^n)\ar[r]_-{d^n_N} & N^{n+1},
    }
    $$
for each $n\in\mathbb{Z}$, which is a direct consequence of the naturality of $\Psi:F\longrightarrow G$. Clearly, by definition, this assignment is functorial.

Similarly, from each natural transformation $\Psi:F\longrightarrow G$ between endofunctors $F, G\in\textup{End}(\mathcal{C})$, we can prove that induces a covariant functor
    $$
    \begin{array}{rcl}
    \mathfrak{R}^*_{\Psi}(\mathcal{C}):\mathfrak{R}^*_F(\mathcal{C})&\longrightarrow & \mathfrak{R}^*_G(\mathcal{C})\\
    M \ \ \  &\longmapsto & \mathfrak{R}^*_{\Psi}(\mathcal{C})(M):=\left(M^n,\Psi_{M^{n+1}}\circ d^n_M\right)_{n\in\mathbb{Z}}\\
   \varphi \ \ \ \ &\longmapsto & \mathfrak{R}^*_{\Psi}(\mathcal{C})(\varphi):=\varphi.
    \end{array}
    $$
In conclusion, we obtain the functors $\mathfrak{L}^*$ and $\mathfrak{R}^*$. The theorem below guarantees that these functors are \emph{embeddings}. According to \cite{AHS}, we consider a functor to be an embedding if it is injective on objects and faithful.
%-----------------------------------%

\begin{theorem}
\label{thm:LR}
Consider an additive category $\mathcal{C}$ and its category of endofunctors $\textup{End}(\mathcal{C})$. The functors
    $$
    \begin{array}{cclccccccl}
    \mathfrak{L}^*:\textup{End}(\mathcal{C}) & \longrightarrow & \mathfrak{L}^*(\mathcal{C}) & & & & & \mathfrak{R}^*:\textup{End}(\mathcal{C})&\longrightarrow & \mathfrak{R}^*(\mathcal{C})\\
    \ \ \ F & \longmapsto & \mathfrak{L}^*_{F}(\mathcal{C}) &&&&& \ \ \ F &\longmapsto & \mathfrak{R}^*_F(\mathcal{C})\\
    \ \ \ \Psi &\longmapsto & \mathfrak{L}^*_{\Psi}(\mathcal{C}) &&&&& \ \ \ \Psi &\longmapsto & \mathfrak{R}^*_{\Psi}(\mathcal{C}). 
    \end{array}
    $$
are a contravariant embedding functor and a covariant embedding functor, respectively.
\end{theorem}

%-----------------------------------%

\begin{proof}
We provide a proof for the functor $\mathfrak{L}^*$, noting that an analogous argument holds for the functor $\mathfrak{R}^*$. Let $\Psi,\Phi\in\textup{Hom}_{\textup{End}(\mathcal{C})}(F,G)$ be two natural transformations such that $\mathfrak{L}^*_{\Psi}(\mathcal{C})=\mathfrak{L}^*_{\Phi}(\mathcal{C})\in\textup{Hom}_{\mathfrak{L}^*(\mathcal{C})}\big(\mathfrak{L}^*_{G}(\mathcal{C}),\mathfrak{L}^*_{F}(\mathcal{C})\big)$. In particular, it follows that $\mathfrak{L}^*_{\Psi}(\mathcal{C})(M)=\mathfrak{L}^*_{\Phi}(\mathcal{C})(M)$, for any object $M\in\mathfrak{L}_{G}(\mathcal{C})$. Following the notation in Theorem \ref{proj&inj}, we obtain that
    $$
    \mathfrak{C}^0\left(\Psi_X\right) =\mathfrak{L}^*_{\Psi}(\mathcal{C})\left(\mathfrak{C}^0\left(1_{G(X)}\right)\right) 
    =\mathfrak{L}^*_{\Phi}(\mathcal{C})\left(\mathfrak{C}^0\left(1_{G(X)}\right)\right)
    =\mathfrak{C}^0\left(\Phi_X\right), \ \ \ \text{for all $X\in\mathcal{C}$.}
    $$
Therefore,  $\Psi_X=\Phi_X$ for all $X\in\mathcal{C}$, which proves this part of the theorem.

Now, suppose $F$ and $G$ are endofunctors on the category $\mathcal{C}$ such that their categories $\mathfrak{L}^*_{F}(\mathcal{C})$ and $\mathfrak{L}^*_{G}(\mathcal{C})$, are equal within the category $\mathfrak{L}^*(\mathcal{C})$. To prove that $F$ and $G$ are actually identical, note that $\mathfrak{C}^0\left(1_{F(X)}\right) \in \mathfrak{L}^*_{F}(\mathcal{C})=\mathfrak{L}^*_{G}(\mathcal{C})$ for all $X\in\mathcal{C}$. Thus, $F(X)=G(X)$ for all $X\in\mathcal{C}$. Furthermore, it follows that $F(f)=G(f)$ for all $f\in\textup{Hom}_{\mathcal{C}}(X,Y)$. Indeed, this equality holds because
$$
    \varphi\in
    \textup{Hom}_{\mathfrak{L}^*_{F}(\mathcal{C})}\left(\mathfrak{C}^0\left(1_{F(X)}\right),\mathfrak{C}^0\left(1_{F(X)}\right)\right)=\textup{Hom}_{\mathfrak{L}^*_{G}(\mathcal{C})}\left(\mathfrak{C}^0\left(1_{F(X)}\right),\mathfrak{C}^0\left(1_{F(X)}\right)\right)
    $$
where $\varphi^n=\delta_{0n}f+\delta_{1n}F(f)$ and $\delta_{ij}$ is the Kronecker delta. This completes the proof of the theorem.

\end{proof}

%-----------------------------------%
For the next proposition, we use Example \ref{example:L_F(C)}(i) which asserts that the category $\mathfrak{L}^*_{\textup{Id}_\mathcal{C}}(\mathcal{C})$ coincides with the category of cochain complexes of $\mathcal{C}$, denoted here by $\textup{Ch}^*(\mathcal{C})$.

\begin{proposition}\label{prop:complexequiv}
If the endofunctor $F:\mathcal{C}\longrightarrow \mathcal{C}$ is an equivalence of categories, then the categories $\textup{Ch}^*(\mathcal{C})$ and $\mathfrak{L}^*_{F}(\mathcal{C})$ are equivalent.
\end{proposition}
\begin{proof}
Suppose that $F:\mathcal{C}\longrightarrow \mathcal{C}$ is an equivalence of categories and let $G:\mathcal{C}\longrightarrow \mathcal{C}$ be a pseudo-inverse of $F$. The equivalence between $\mathfrak{L}^*_{F}(\mathcal{C})$ and $\textup{Ch}^*(\mathcal{C})$ is given by:
\begin{equation}\label{eq:complexequiv}
\begin{array}{rcl}
    \mathcal{V}: \mathfrak{L}^*_F(\mathcal{C}) &\longrightarrow & \textup{Ch}^*(\mathcal{C})\\
    M=(M^n,d_M^n)_{n\in\mathbb{Z}} \ \ \  &\longmapsto & \mathcal{V}(M):=\left(G^{n}(M^n),G^{n+1}(d_{M}^n)\right)_{n\in\mathbb{Z}}\\
    \varphi=(\varphi^n)_{n\in\mathbb{Z}} \ \ \ \ &\longmapsto & \mathcal{V}(\varphi):=(G^n(\varphi^n))_{n\in\mathbb{Z}}
    \end{array}
\end{equation}
Indeed, if we denote in (\ref{eq:complexequiv}), $\mathcal{M}^n:=G^{n}(M^n)$ and $\partial_{\mathcal{M}}^n:=G^{n+1}(d_{M}^n)$, the equivalence $d_M^{n+1}F(d_M^{n})=0$ if and only if $G(d_M^{n+1})d_M^n=0$, implies that $\partial_{\mathcal{M}}^{n+1}\partial_L^n=0$ for all $n\in\mathbb{Z}$. Likewise, by setting in (\ref{eq:complexequiv}) $\psi^n:=G^n(\varphi^n)$, it is easy to verify that $\psi^{n+1}\partial_{\mathcal{M}}^{n}=\partial_{\mathcal{N}}^n\psi^n$, using the relation $\varphi^{n+1}d_M^n=d_N^{n}F(\varphi^n)$ for a morphism $\varphi^n:M^n\rightarrow N^n$ for all $n\in\mathbb{Z}$. By definition, it is now straightforward to conclude that $\mathcal{V}$ is an equivalence of categories, which completes the proof.
\end{proof}

%-----------------------------------------------%

\begin{remark}\label{rmk:reflectiso}\rm

\noindent
\begin{enumerate}[(i)]
\item A proof analogous to that for Proposition \ref{prop:complexequiv} shows an equivalence between $\mathfrak{R}^*_{F}(\mathcal{C})$ and $\textup{Ch}^*(\mathcal{C})$ when $F:\mathcal{C}\longrightarrow\mathcal{C}$ is an equivalence.
\item Theorem \ref{thm:LR} naturally raises the question: Does the functor  $\mathfrak{L}^{\ast}$ \emph{reflect isomorphisms}? That is, if $\mathfrak{L}_F^{\ast}(\mathcal{C})\cong \mathfrak{L}_G^{\ast}(\mathcal{C})$, must $F$ be equivalent (or isomorphic) to $G$? The answer is, in general, no. Indeed, as shown in Proposition \ref{prop:complexequiv}, we have  $\mathfrak{L}_F^{\ast}(\mathcal{C})\cong \mathfrak{L}_{\textup{Id}_\mathcal{C}}^{\ast}(\mathcal{C})$, but $F$ is not necessarily naturally equivalent (or isomorphic) to $\textup{Id}_\mathcal{C}$. A concrete example of this is when $\mathcal{C}=\Bbbk\textup{-mod}\times\Bbbk\textup{-mod}$, with $\Bbbk$ a field, and $F:\mathcal{C}\longrightarrow\mathcal{C}$ defined by $(V,W)\longmapsto (W,V)$. Clearly, there is no a natural equivalence between $F$ and $\textup{Id}_\mathcal{C}$. This shows that the functor $\mathfrak{L}^{\ast}$ can map different objects in the category $\textup{End}(\mathcal{C})$ to categories that are equivalent (or even isomorphic) within $\mathfrak{L}^{\ast}(\mathcal{C})$.
\end{enumerate}
\end{remark}
%-----------------------------------%

The following result is inspired by the description given by Happel in \cite{H88} for the category of modules over a repetitive algebra $\widehat{A}$, where $A$ is an Artin algebra. In his work, Happel introduced its objects using both the tensor product and the Hom functor, presenting them as indistinguishable objects. This description allowed him to study several properties of the category, such as the fact that it is a Frobenius category (see \cite[Chap. 2 \S 2, \S 3]{H88}). It also served as a key tool to prove the important relation between the category of stable finitely generated modules over the repetitive algebra and the bounded derived category of finitely generated $A$-modules (see \cite[Chap. 2, \S 4]{H88}).

\begin{theorem}
\label{thm:adjoint}
Let $(F,G)$ be an adjoint pair of endofunctors over $\mathcal{C}$ and $\varepsilon:FG\longrightarrow \textup{Id}_{\mathcal{C}}$ its counit. Then, there exists an isomorphism of categories
     $$
    \begin{array}{rcl}
    \mathcal{E}: \mathfrak{L}^*_F(\mathcal{C}) &\longrightarrow & \mathfrak{R}^*_G(\mathcal{C})\\
    M \ \ \  &\longmapsto & \mathcal{E}(M):=\left(M^n,d_{\mathcal{E}(M)}^n\right)_{n\in\mathbb{Z}}\\
    \varphi \ \ \ \ &\longmapsto & \mathcal{E}(\varphi):=\varphi
    \end{array}
    $$
where $d_{\mathcal{E}(M)}^n\in \textup{Hom}_{\mathcal{C}}\left(M^n,G(M^{n+1})\right)$ is the unique morphism such that $\varepsilon_{M^{n+1}}F(d_{\mathcal{E}(M)}^n)=d_{M}^n$.
\end{theorem}
%--------------------------------%

\begin{proof}
Firstly, by adjointness, for each object $M\in\mathfrak{L}^*_{F}(\mathcal{C})$, the morphisms $d_{M}^n\in \textup{Hom}_{\mathcal{C}}\left(F(M^n),M^{n+1}\right)$ and $d_{\mathcal{E}(M)}^n\in \textup{Hom}_{\mathcal{C}}\left(M^n,G(M^{n+1})\right)$ are the unique morphisms such that 
    \begin{equation}
    \label{eq:adj}
    \varepsilon_{M^{n+1}}\circ F(d_{\mathcal{E}(M)}^n)=d_{M}^n\ \ \ \text{and}\ \ \ \ G(d_{M}^n)\circ \eta_{M^{n}}=d_{\mathcal{E}(M)}^n,\ \ \ \ \text{for all $n\in\mathbb{Z}$}    
    \end{equation}
where $\eta:\textup{Id}_{\mathcal{C}}\longrightarrow GF$ and  $\varepsilon:FG\longrightarrow \textup{Id}_{\mathcal{C}}$ are the unit and counit, respectively.    

For every object $M\in\mathfrak{L}^*_{F}(\mathcal{C})$, the hypothesis $d_{M}^{n+1}\circ F(d_{M}^{n})=0$, for all $n\in\mathbb{Z}$, and the equation \eqref{eq:adj} imply that $\varepsilon_{M^{n+2}}\circ F\big(d_{\mathcal{E}(M)}^{n+1}\circ d_{M}^{n}\big)=d_{M}^{n+1}\circ F(d_{M}^{n})=0$, and by adjointness  $d_{\mathcal{E}(M)}^{n+1}\circ d_{M}^{n}=0$. Hence, from equation \eqref{eq:adj} we obtain $G(d_{\mathcal{E}(M)}^{n+1})\circ d_{\mathcal{E}(M)}^n=G(d_{\mathcal{E}(M)}^{n+1}\circ d_{M}^n)\circ \eta_{M^{n}}=0$, for all $n\in \mathbb{Z}$. Therefore, $\mathcal{E}(M)\in\mathfrak{R}^*_{G}(\mathcal{C})$.

Now, let $\varphi:M\longrightarrow N$ be a morphism in $\mathfrak{L}^*_{F}(\mathcal{C})$. Since $d_N^n\circ F(\varphi^n)=\varphi^{n+1}\circ d_M^n$ for all $n\in\mathbb{Z}$, from equation \eqref{eq:adj} we obtain $\varepsilon_{N^{n+1}}\circ F(d_{\mathcal{E}(N)}^n\circ \varphi^n)=\varphi^{n+1}\circ d_M^n$, it follows by adjointness $d_{\mathcal{E}(N)}^n\circ \varphi^n=G(\varphi^{n+1}\circ d_M^n)\circ\eta_{M^n}=G(\varphi^{n+1})\circ G(d_M^n)\circ\eta_{M^n}$, which implies that $d_{\mathcal{E}(N)}^n\circ \varphi^n=G(\varphi^{n+1})\circ d_{\mathcal{E}(M)}^n$ for all $n\in\mathbb{Z}$, from equation \eqref{eq:adj}. Thus, $\mathcal{E}(\varphi)=\varphi:M\longrightarrow N$ is a morphism in $\mathfrak{R}^*_{G}(\mathcal{C})$.

Since the definition of $\mathcal{E}$ depends uniquely and naturally on the properties of adjointness, it is easy to check that the functor  
    $$
    \begin{array}{rcl}
    \mathcal{R}: \mathfrak{R}^*_G(\mathcal{C}) &\longrightarrow & \mathfrak{L}^*_F(\mathcal{C})\\
    M \ \ \  &\longmapsto & \mathcal{R}(M):=\left(M^n,d_{\mathcal{R}(M)}^n\right)_{n\in\mathbb{Z}}\\
    \varphi \ \ \ \ &\longmapsto & \mathcal{R}(\varphi):=\varphi
    \end{array}
    $$
where $d_{\mathcal{R}(M)}^n\in \textup{Hom}_{\mathcal{C}}\left(F(M^n),M^{n+1}\right)$ is the unique morphism such that $G(d_{\mathcal{R}(M)}^n)\circ \eta_{M^{n}}=d_{M}^n$, is well-defined, and satisfies $\mathcal{E}\circ\mathcal{R}=\textup{Id}_{\mathfrak{R}^*_G(\mathcal{C})}$ and $\mathcal{R}\circ\mathcal{E}=\textup{Id}_{\mathfrak{L}^*_F(\mathcal{C})}$. This completes the proof of the theorem.

\end{proof}

\begin{remark}\rm
\label{rem:falgebras}

\noindent

\begin{enumerate}[(i)]
\item As established in the proof of Theorem \ref{thm:graduate}, the category $\mathfrak{L}^{\ast}(\mathcal{C})$, which is actually a 2-category, can be viewed as a collection of ``graded'' Eilenberg-Moore categories. The results of this section allow us to study how  End$(\mathcal{C})$  interacts with these categories, as we explain below.

To illustrate these interactions, we introduce the 2-category $\textup{Gr}^*(\mathcal{C})$, as a subcategory of $\mathfrak{L}^\ast(\mathcal{C})$. Its objects are Eilenberg-Moore categories $\mathfrak{L}_F^{\ast}(\mathcal{C})$, that satisfy the two conditions in step two of the proof of Theorem \ref{thm:graduate}. Its 1-morphisms are the functors $\mathfrak{L}^*_{\Psi}(\mathcal{C}):\mathfrak{L}^*_G(\mathcal{C})\longrightarrow\mathfrak{L}^*_F(\mathcal{C})$, which are induced by natural transformations $\Psi: F\rightarrow G$ between endofunctors.

The material developed in this section leads to two results concerning the aforementioned relationship between End$(\mathcal{C})$ and $\text{Gr}^*(\mathcal{C})$:
\begin{itemize}
\item $\text{Gr}^*(\mathcal{C})$ can be endowed with a monoidal structure derived from End$(\mathcal{C})$.
\item For any $F\in\text{Aut}(\mathcal{C})$, the \emph{categorical group} (cf. \cite[Definition 2.11.4, p. 44]{EGNO15}) of End$(\mathcal{C})$ of the equivalences over $\mathcal{C}$, is in correspondence with the unity (under the monoidal structure above) $\textup{Ch}^*(\mathcal{C})$.
\end{itemize}

For the first item, it is sufficient to consider the bifunctor
\begin{eqnarray*}
\widetilde{\otimes}: \textup{Gr}^*(\mathcal{C})\times\textup{Gr}^*(\mathcal{C})&\longrightarrow & \textup{Gr}^*(\mathcal{C})\\
\left(\mathfrak{L}^*_F(\mathcal{C}),\mathfrak{L}^*_G(\mathcal{C})\right) &\longmapsto & \mathfrak{L}^*_F(\mathcal{C})\widetilde{\otimes}\mathfrak{L}^*_G(\mathcal{C}):=\mathfrak{L}^*_{F\circ G}(\mathcal{C})\\
\left(\mathfrak{L}^*_{\alpha}(\mathcal{C}),\mathfrak{L}^*_{\beta}(\mathcal{C})\right) &\longmapsto & \mathfrak{L}^*_{\alpha}(\mathcal{C})\widetilde{\otimes}\mathfrak{L}^*_{\beta}(\mathcal{C})
\end{eqnarray*}
where the monoidal product of the functors $\mathfrak{L}^*_{\alpha}(\mathcal{C}):\mathfrak{L}^*_{G}(\mathcal{C})\longrightarrow \mathfrak{L}^*_F(\mathcal{C})$ and $\mathfrak{L}^*_{\beta}(\mathcal{C}):\mathfrak{L}^*_{G'}(\mathcal{C})\longrightarrow \mathfrak{L}^*_{F'}(\mathcal{C})$ is the functor $\mathfrak{L}^*_{\alpha}(\mathcal{C})\widetilde{\otimes}\mathfrak{L}^*_{\beta}(\mathcal{C}):\mathfrak{L}^*_{G}(\mathcal{C})\widetilde{\otimes}\mathfrak{L}^*_{G'}(\mathcal{C})\longrightarrow \mathfrak{L}^*_F(\mathcal{C})\widetilde{\otimes}\mathfrak{L}^*_{F'}(\mathcal{C})$ which is defined by the composition of functors $\mathfrak{L}^*_{F(\beta)}(\mathcal{C})\circ\mathfrak{L}^*_{\alpha_{G'}}(\mathcal{C})$ which in turn is defined by the natural transformations 
$$
     F(\beta)=1_{F}\otimes\beta: F\circ F'\longrightarrow F\circ G'\ \ \ \ \text{and}\ \ \ \ \alpha_{G'}=\alpha\otimes 1_{G'}: F\circ G'\longrightarrow G\circ G'.
     $$
coming from of monoidal structure over End$(\mathcal{C})$. In this case, observe that the unity $\boldsymbol{1}_{\widetilde{\otimes}}\in\textup{Gr}^*(\mathcal{C})$ is the category of cochain complexes on $\mathcal{C}$, that is, $\boldsymbol{1}_{\widetilde{\otimes}}=\mathfrak{L}^*_{\textup{Id}_{\mathcal{C}}}(\mathcal{C})=\textup{Ch}^*(\mathcal{C})$.

The second item is an immediate consequence of the construction above, Proposition \ref{prop:complexequiv} and Theorem \ref{thm:LR}.

\item The 2-category $\textup{Gr}^*(\mathcal{C})$ possess a subtle but interesting property that connects to Example \ref{example:L_F(C)}. The zero endofunctor, $F=0$, induces the zero object in  $\textup{Gr}^*(\mathcal{C})$ (regarded as a 2-category). Indeed, observe that zero endofunctor is the zero object of End$(\mathcal{C})$. This fact, along with the universal property of a zero object, provides two canonical natural transformations (and unique up to isomorphisms) for any $G\in\textup{End}(\mathcal{C})$,  $\mu:0\longrightarrow G$ and $\beta:G\longrightarrow 0$. As developed in this section, these transformations induce two canonical functors, $\mathfrak{L}_{\mu}^{\ast}(\mathcal{C}): \mathfrak{L}_G^{\ast}(\mathcal{C})\longrightarrow\mathfrak{L}_0^{\ast}(\mathcal{C})$ and $\mathfrak{L}_{\beta}^{\ast}(\mathcal{C}):\mathfrak{L}_0^{\ast}(\mathcal{C})\longrightarrow \mathfrak{L}_G^{\ast}(\mathcal{C})$. Consequently, for any pair of endofunctors $F, G\in\textup{End}(\mathcal{C})$ the diagram below
   \begin{equation*}\label{eq:grc}
     \xymatrix{
 & \mathfrak{L}^*_G(\mathcal{C}) \ar[r]^{\mathfrak{L}^*_{\Psi}(\mathcal{C})}& \mathfrak{L}^*_F(\mathcal{C})\ar[d]^{\mathfrak{L}_{\mu}^{\ast}(\mathcal{C})} & \\
 &\mathfrak{L}_0^{\ast}(\mathcal{C})\ar[r]_{\textup{Id}_{\mathfrak{L}_0^{\ast}(\mathcal{C})}}\ar[u]^{\mathfrak{L}_{\beta}^{\ast}(\mathcal{C})}&\mathfrak{L}_0^{\ast}(\mathcal{C})& 
 }
 \end{equation*} 
 is commutative, where the functor
 $\mathfrak{L}^*_{\Psi}(\mathcal{C}):\mathfrak{L}^*_G(\mathcal{C})\longrightarrow\mathfrak{L}^*_F(\mathcal{C})$ is induced by a natural transformation $\Psi: F\longrightarrow G$ between the endofunctors $F$ and $G$ on $\mathcal{C}$. In consequence, the categories of 1-morphisms $(\mathfrak{L}^*_G(\mathcal{C}), \mathfrak{L}^*_0(\mathcal{C}))$ and $(\mathfrak{L}^*_0(\mathcal{C}),\mathfrak{L}^*_G(\mathcal{C}))$ each contain a single isoclass. This implies that $\mathfrak{L}^*_0(\mathcal{C})$ is the zero object in the 2-category $\textup{Gr}^*(\mathcal{C})$, an object that is also endowed with the zero (endo)functor $\mathfrak{L}_{\mathcal{O}}^{\ast}(\mathcal{C}):=\mathfrak{L}_{\mu}^{\ast}(\mathcal{C})\circ\mathfrak{L}_{\beta}^{\ast}(\mathcal{C})$.
\end{enumerate}
\end{remark}

%======================================%
            % SOME APPLICATIONS %
%======================================%

\section{Some applications}\label{sec:apply}
Along this section we will fix a field $\Bbbk$, a finite dimension $\Bbbk$-algebra $A$ (with unity) and all the endofunctors are $\Bbbk$-linear. In this section we explore consequences of our constructions, focusing on two key settings as are representation theory of algebras and algebraic geometry with implications in both contexts.

Firstly, we will study $\otimes$-representable functors over the category of finitely generated left modules for a $\Bbbk$-algebra $A$, which we denote as $A$\textup{-mod}. This characterization provides results and generalizations for algebra representation theory, particularly in the context of ``generalized repetitive algebras''.

Secondly, we translate these characterizations and consequences to the study of endofunctors on the category of quasi-coherent sheaves over a specific class of schemes. We can interpret this as a ``global characterization'' of  $\otimes$-representable functors in the commutative case (see introduction to Subsection \ref{geom_app}). Our interest in these connections on categories of sheaves stems from their potential to offer new interactions with constructions in 2-algebraic affine geometry, as proposed in \cite{BC}, \cite{BC2}.
%--------------------------------------%

\subsection{(Locally) Finite categories}\label{subs:locfinite}
In this subsection, we will study the class of $\Bbbk$-linear $\otimes$-representable endofunctors within the monoidal category $\textup{End}(A\text{-}\textup{mod})$. From this point on, we will use the notation $\textup{End}_{e}(A\textup{-\textup{mod}})$ and $\textup{End}_{re}(A\text{-\textup{mod}})$ to refer to the (rigid) monoidal categories of exact and right-exact endofunctors over $A\text{-\textup{mod}}$, respectively. All the results presented here have implications for Subsections \ref{subsec:GRA} and \ref{geom_app}.

%-----------------------------------

\begin{definition}[{\cite[Definition 1.8.8, p. 10]{EGNO15}}]
An endofunctor $F\in\textup{End}(A\textup{-mod})$ is called \emph{$\otimes$-representable} if there exists an $A$-bimodule $D$ such that $F$ is naturally isomorphic to $D\otimes_A -$.
    \end{definition}
%-----------------------------------%

By \cite[Example 2.5.3, p. 32]{EGNO15} we have a natural monoidal functor:
\begin{eqnarray*}
\label{eq:bi-end}
\mathcal{T}:A\textup{-bimod}&\longrightarrow & \textup{End}(A\textup{-mod})\\
D \ \ \  &\longmapsto & D\otimes_A -
\end{eqnarray*}
where $A$\text{-bimod} is the strict monoidal category of finitely generated $A$-bimodules with monoidal structure given by the usual tensor product $\otimes_A$ whose unity is $A$.

As guarantees the following proposition, the functor $\mathcal{T}:A\textup{-bimod}\longrightarrow \textup{End}(A\textup{-mod})$ induces a monoidal equivalence between certain monoidal categories.

%-----------------------------------%
\begin{proposition}
\label{prop:rightrep}
$F\in\textup{End}(A\textup{-mod})$ is $\otimes$-representable endofunctor if and only if $F\in\textup{End}_{re}(A\textup{-mod})$. In particular, the functor $\mathcal{T}:A\textup{-bimod}\longrightarrow \textup{End}(A\textup{-mod})$ establishes an equivalence between the monoidal categories
\begin{eqnarray*}
\mathcal{T}:A\textup{-bimod}&\longrightarrow & \textup{End}_{re}(A\textup{-mod})\\
D \ \ \  &\longmapsto & D\otimes_A -
\end{eqnarray*}
\end{proposition}
%-------------------------------%

\begin{proof}
Clearly, if $F$ is isomorphic to $D\otimes_A -$, then $F$ is a right exact functor. Conversely, suppose that $F$ is a right exact functor and consider the $A$-bimodule $D=F(A)$. Since the proof in this direction involve ideas already established in \cite[Proposition 1.8.10, p. 10 and Proposition 2.5.4, p. 32]{EGNO15}, we omit a detailed proof here.

\end{proof}

%-------------------------------%

\begin{remark}\rm
As a consequence of Proposition \ref{prop:rightrep} every $F\in \textup{End}_{re}(A\textup{-mod})$ has a right adjoint given by the (covariant representable left-exact) endofunctor $G=\text{Hom}_A(D,-)$, where $D$ is the $A$-bimodule representing $F$. Thus, $(F,G)$ is an adjoint pair on $A$-mod. In particular, the category $\mathfrak{L}^*_F(A\textup{-mod})$ is abelian, complete and co-complete, by Corollary \ref{cor:adjconsq}. 
\end{remark}

%-------------------------------%

In the special case where $F$ is an exact endofunctor, i.e., $F\in\textup{End}_{e}(A\textup{-\textup{mod}})\hookrightarrow \textup{End}_{re}(A\text{-\textup{mod}})$, we can obtain a new characterization for the object $D$ that $\otimes$-represents $F$ as claimed below. This characterization is guaranteed by Proposition \ref{prop:rightrep} (cf. \cite[Exercise 2.10.16, p. 43]{EGNO15}).

%------------------------%

\begin{proposition}
\label{prop:exactcase}
Let $F\in\textup{End}_{re}(A\text{-\textup{mod}})$ be an endofunctor $\otimes$-represented by $D\in A\textup{-bimod}$. Then, $F\in\textup{End}_{e}(A\textup{-mod})$ if and only if $D$ is projective as left and right $A$-module.
\end{proposition}

%-------------------------%

\begin{proof}
Due to the fact that $F$ is naturally isomorphic to $D\otimes_A -$ we obtain that, $F$ is left exact if and only if $D$ is a flat left (right) $A$-module. Since $A$ is a perfect ring it follows that, $D$ is a flat left (right) $A$-module if and only if $D$ is a projective left (right) $A$-module (see e.g \cite[Theorem 28.4, p. 315]{Arings}), which completes the proof.

\end{proof}

\begin{remark}\rm
\noindent

\begin{enumerate}[(i)]
% \item From Proposition \ref{prop:exactcase}, we obtain a new equivalence of strict monoidal categories induced by the functor $\mathcal{T}:A\textup{-bimod}\longrightarrow \textup{End}_{re}(A\textup{-mod})$, specifically 
% \begin{eqnarray*}
% \mathcal{T}:A\textup{-projbimod}&\longrightarrow & \textup{End}_{e}(A\textup{-mod})\\
% D \ \ \  &\longmapsto & D\otimes_A -
% \end{eqnarray*}
% where $A\textup{-projbimod}$ corresponds to the full subcategory of projective finitely generated $A$-bi\-modules.

\item Let $D=\Hom_{\k}(-,\Bbbk)$ be the standard duality on $A\textup{-mod}$. Recall that the Nakayama functors $\mathcal{N}_A$ and $\mathcal{N}_A^{-1}$ are naturally isomorphic to the endofunctors $D(A)\otimes_A -$ and $\text{Hom}_A(D(A),-)$, respectively (see \cite[Proposition 5.2, Chapter III, p. 248]{SY}). In this case also we obtain a mutually inverse equivalences of categories between projectives $A\textup{-proj}$ and injectives $A\textup{-inj}$ in $A$-\textup{mod}
    $$
    A\textup{-proj}\begin{array}{c}\stackrel{\mathcal{N}_A}{\longrightarrow}\\ \stackrel{\longleftarrow}{\scriptstyle{\mathcal{N}_A^{-1}}}\end{array}A\textup{-inj}
    $$
This equivalence is used by Happel to show that $\widehat{A}$-mod, the category of finitely generated modules over the repetitive algebra of $A$, is a Frobenius category. More details in \cite[p. 62]{H88}. This result is the motivation for the proposition below.
\end{enumerate}
\end{remark}

%---------------------------------%

\begin{proposition}
\label{prop:proj-inj}
Let $I$ be a nonzero finitely generated $A$-bimodule which is injective as left and rigth $A$-module. Then, the endofunctors $F=I\otimes_A -$ and $G=\textup{Hom}_A(I,-)$ are mutually inverse equivalences when restricted to the categories $A\textup{-proj}$ and $A\textup{-inj}$, respectively.
\end{proposition}
\begin{proof}
Clearly, putting $A$ the minimal projective generator $A$-module, we have that $F(A)=I$ is an injective $A$-module. In consequence, if $L=P\oplus M$, where $L$ is a free $A$-module and $P$ is a projective $A$-module, it follows that $F(L)=F(P)\oplus F(M)$, where $F(L)$ is injective. Hence, $F(P)$ is an injective $A$-module. Conversely, the argument is similar as above putting $D(A)$ the minimal injective cogenerator $A$-bimodule to obtain that $G(D(A))=D(I)$, which is a projective $A$-module.

\end{proof}

%-------------------------------%

\subsection{Generalized Repetitive algebras or Cochain complexes over \texorpdfstring{$\otimes$}{eq:tensor}-representable endofunctors}\label{subsec:GRA}

Throughout this subsection, we fix an algebraically closed field $\Bbbk$ of characteristic zero. We let $D=\Hom_{\k}(-,\Bbbk)$ be the standard duality on $A$-\textup{mod}. Our focus is on finitely generated left modules. Although some of our results can be extended, for clarity and concreteness we will restrict our attention to finitely generated modules and algebras over the field $\Bbbk$.

From the results in Subsection \ref{subs:locfinite}, we now explore a special case of the categories $\mathfrak{L}^*_F(A\textup{-mod})$ and $\mathfrak{R}^*_F(A\textup{-mod})$, where $F$ is a $\otimes$-representable endofunctor on the category of finitely generated $A$-modules. These categories have received considerable attention in several mathematical fields due to their diverse applications. Some key examples include:
\begin{itemize}
\item $\textup{Ch}^*(A\textup{-Mod})=\mathfrak{L}^*_{\textup{Id}_{A-\textup{Mod}}}(A\textup{-Mod})$.
\item $\mathfrak{L}^b_F(A\textup{-mod})\cong\widehat{A}\textup{-mod}$, when $F\cong D(A)\otimes_A-$ and $\widehat{A}$ the repetitive algebra of $A$ (cf. \cite{HW83}).
\item Proposition \ref{B[A]functors1} allows us to state that the category $\mathfrak{L}^b_F(A\textup{-mod})$ is equivalent to the category of modules of the \emph{Cluster repetitive algebra} of $A$. This holds when $F$ is defined as $E\otimes_A-$ (with $E=\textup{Ext}^2_A(D(A),A)$), a topic studied in \cite[Section 3.3, p. 153]{Ass18} and, when $A$ is tilted algebra, in \cite[Section 1.3]{ABS09}.
\item The study of equivalences between stable module categories of repetitive algebras via Wakamatsu-Tilting modules, see \cite{w21}.
\end{itemize}

\vspace{0,5cm}

Inspired by the work of Hughes and Waschb\"usch (cf. \cite{HW83}), the aim of this subsection is to extend the construction of the repetitive algebra for a finite-dimensional $\Bbbk$-algebra $A$. Our approach follows the ideas from \cite[Proposition 3]{Gir18} for this purpose (see Proposition \ref{B[A]functors1}). Namely, we focus on a finitely generated $A$-bimodule $B$ and we define the algebra $\widehat{B}$ as the infinite dimensional $\Bbbk$-algebra (without unit):
	$$
	\widehat{B}:= \left(\bigoplus_{n\in\mathbb{Z}}A\right)\oplus\left(\bigoplus_{n\in\mathbb{Z}}B\right).
	$$
whose multiplication rule, for $(a_n,b_n)_{n\in\mathbb{Z}}, (a'_n,b'_n)_{n\in\mathbb{Z}}\in \widehat{B}$, is given by:
	$$
	(a_n,b_n)_{n\in\mathbb{Z}} \cdot(a'_n,b'_n)_{n\in\mathbb{Z}} = (a_na'_n, a_{n+1}b'_n +b_na'_n)_{n\in\mathbb{Z}}.
	$$
Here, an element $(a_n,b_n)_{n\in\mathbb{Z}}$ of $\widehat{B}$ has only finitely many of the $a_n$'s and $b_n$'s being non-zero. Now, we write $\widehat{e}_n:=\left(\delta_{nm},0\right)_{m\in\mathbb{Z}}\in\widehat{B}$ for each $n\in\mathbb{Z}$, where $\delta_{nm}$ is the Kronecker delta. The set $\left\{\widehat{e}_n \mid n\in\mathbb{Z}\right\}$ is a complete set of pairwise orthogonal idempotents in $\widehat{B}$, in the sense that, the following properties hold:
\begin{itemize}
	\item $\widehat{e}_n\widehat{e}_m=\delta_{nm}\widehat{e}_n$ for all $n,m\in\mathbb{Z}$.
	\item For all $\widehat{x}=\left(a_n,b_n\right)_{n\in\mathbb{Z}}\in\widehat{B}$ and $s,m\in\mathbb{Z}$ 
	\begin{displaymath}
	\widehat{e}_{s}\widehat{x}\widehat{e}_m = 
	\begin{cases}
	\left(\delta_{nm}a_n,0\right)_{n\in\mathbb{Z}}, & \text{if $s=m$;}\\
	\left(0,\delta_{nm}b_n\right)_{n\in\mathbb{Z}}, & \text{if $s=m+1$;}\\
	\left(0,0\right)_{n\in\mathbb{Z}}, & \text{otherwise.}
	\end{cases}
	\end{displaymath}
Moreover, for each $m\in\mathbb{Z}$, we obtain the isomorphisms of $\Bbbk$-algebras $\theta_m$ and, likewise, the collection of isomorphisms of $A$-bimodules $\phi_m$, both determined as follows:
	$$
	\ \ \ \ \ \ \ \ 
    \begin{array}{rcl}
	\theta_m: A & \longrightarrow & \widehat{e}_{m}\widehat{B}\widehat{e}_m\\
    		  a\,\,  & \longmapsto     & \theta_m(a):=(a\delta_{nm},0)_{n\in\mathbb{Z}}
    \end{array}
    \ \ \ \ \ \ \ \ \ \text{and}
    \ \ \ \ \ \ \ \ \ 
    \begin{array}{rcl}
	\phi_m: B & \longrightarrow & \widehat{e}_{m+1}\widehat{B}\widehat{e}_m\\
    		b\,\, & \longmapsto     & \phi_m(b):=(0,b\delta_{nm})_{n\in\mathbb{Z}}
    \end{array}
    $$

\item Finally,
    $$
    \widehat{B}=\left(\bigoplus\limits_{n\in\mathbb{Z}}\widehat{e}_{n}\widehat{B}\widehat{e}_n\right)\oplus\left(\bigoplus\limits_{n\in\mathbb{Z}}\widehat{e}_{n+1}\widehat{B}\widehat{e}_n\right).
	$$		
\end{itemize}
Therefore $\widehat{B}$ can be interpreted as an ``infinite matrix algebra" as follows:
	$$
	\widehat{B}=\begin{bmatrix}
    \ddots &  	&   &   	 & 		  \\
    \ddots & A  &   &   	 & 		  \\
           & B  & A &   	  &        \\
           &    & B &   A 	  & 	   \\
           &    &   & \ddots & \ddots \\
	\end{bmatrix}
	$$
when the product in $\widehat{B}$ of two infinite matrices of this form is the usual product of matrices, except that the entries of the second lower diagonal in the product are ignored and set to be zero entries.

%-----------------------------------------------------------%

\begin{remark}\rm
In this paper, we denote $\widehat{D(A)}$ as the repetitive algebra of an algebra $A$, following the construction introduced above. This concept, introduced by Hughes and Waschb\"usch in their work \cite{HW83}, was originally denoted by $\widehat{A}$. Although our notation is inspired by theirs, there are slight differences between our approach and the original definition of repetitive algebras.
\end{remark}

%-----------------------------------------------------------%

\begin{lemma}
\label{lem:PedroGerman}
If $M\in\widehat{B}$\textup{-mod}, then $M$ decomposes (as vector space) in the form $M=\bigoplus\limits_{n\in\mathbb{Z}}\widehat{e}_nM$.
\end{lemma}

%-----------------------------------------------------------%

\begin{proof}
Suppose that $\left\{m_1,m_2,\dots,m_t\right\}$ is the set of generators of $M$ as $\widehat{B}$-module. Clearly, each $m\in M$ has the form $m=\sum\limits_{i=1}^t\widehat{b}_im_i$ with $\widehat{b}_1,\widehat{b}_2,\dots,\widehat{b}_t\in \widehat{B}$. Since $\widehat{B}=\bigoplus\limits_{n\in\mathbb{Z}}\widehat{e}_n\widehat{B}$ we have $\widehat{b}_i=\sum\limits_{n\in\mathbb{Z}}\widehat{e}_n\widehat{b}_{in}$, where $\widehat{b}_{in}\in \widehat{B}$. This implies that
	$$
	m=\sum_{i=1}^t\left(\sum_{n\in\mathbb{Z}}\widehat{e}_n\widehat{b}_{in}\right)m_i=\sum_{n\in\mathbb{Z}}\widehat{e}_n
\left(\sum_{i=1}^t\widehat{b}_{in}m_i\right)\in\bigoplus_{n\in\mathbb{Z}}\widehat{e}_n M.
	$$

\end{proof}

%-----------------------------------------------------------%
A key consequence of Lemma \ref{lem:PedroGerman} is that it assures the existence of the following covariant functor
$\mathcal{F}:\widehat{B}\textup{-mod}
\longrightarrow\mathfrak{L}^b_{B\otimes_A-}(A\textup{-mod})$
 defined as follows: For any $M=\bigoplus\limits_{n\in\mathbb{Z}}\widehat{e}_nM\in\widehat{B}\textup{-mod}$, denote by $\mathcal{F}(M)$ the sequence $\left(M^n,d^n_M\right)_{n\in\mathbb{Z}}$ determined by:
\begin{itemize}
\item $A$-modules $M^n:=\widehat{e}_nM$, for each $n\in\mathbb{Z}$. Such a structure of $A$-module is consequence of the isomorphism of $\Bbbk$-algebras $\widehat{e}_{n}\widehat{B}\widehat{e}_n\cong A$, whose action is given by 
	$$
	\left(\widehat{e}_{n}\widehat{b} \widehat{e}_n\right)\cdot\widehat{e}_n v:=\widehat{e}_{n}\widehat{b}\widehat{e}_nv,\ \ \ \text{for all } \widehat{b}\in\widehat{B},\ v\in M;
	$$
\item $A$-homomorphisms, for each $n\in\mathbb{Z}$
    $$
	\begin{array}{rccl}
	d^n_M:&B\otimes_A M^n & \longrightarrow & M^{n+1}\\
    &(\widehat{e}_{n+1}\widehat{b}\widehat{e}_n)\otimes\widehat{e}_n v & \longmapsto     & d^n_M\left((\widehat{e}_{n+1}\widehat{b}\widehat{e}_n)\otimes\widehat{e}_n v\right):=\widehat{e}_{n+1}\widehat{b}\widehat{e}_n v
    \end{array}
    $$
which are well-defined due to the isomorphisms of $A$-bimodules $\widehat{e}_{m+1}\widehat{B}\widehat{e}_m\cong B$. Also, it is easy to check that $d^{n+1}_M(1_B\otimes d^n_M)=0$, for each $n\in\mathbb{Z}$, where $1_B:B\longrightarrow B$ is the identity map. 
\end{itemize}

Now, for any $\widehat{B}$-homomorphism $g:M\longrightarrow N$ we denote by $\mathcal{F}(g)$ the sequence $\left(g^n\right)_{n\in\mathbb{Z}}$ of $A$-homomorphisms 
    $$
	\begin{array}{rccl}
	g^n: & M^n & \longrightarrow & N^{n}\\
         & \widehat{e}_n v & \longmapsto     & g^n\left(\widehat{e}_nv\right):=\widehat{e}_ng(v)
    \end{array}
    $$
%-----------------------------------------------------------%
%-----------------------------------------------------------%

% %-----------------------------------------------------------%

%-----------------------------------------------------------%

Hughes and Waschb\"usch in \cite{HW83} state that $\mathfrak{L}^b_{D(A)\otimes_A -}(A\textup{-mod})$ is the ca\-te\-go\-ry of $\widehat{D(A)}$-mod. In this subsection we will generalize that result, following the ideas of Giraldo in \cite{Gir18}. More precisely, we show that the categories $\mathfrak{L}^b_{B\otimes_A -}(A\textup{-mod})$ and $\widehat{B}$-mod are equivalent. To this end, we consider the covariant functor 
    \begin{equation}\label{equiv:basic}
      \begin{array}{ccl}
    \mathcal{G}:\mathfrak{L}^b_{B\otimes_A-}(A\textup{-mod}) & \longrightarrow & \widehat{B}\textup{-mod} \\
    \ \ \ \ \ M &\longmapsto & \mathcal{G}(M):=\bigoplus\limits_{n\in\mathbb{Z}}M^n\\ 
    \ \ \ \ \ \varphi &\longmapsto & \mathcal{G}(\varphi):=\bigoplus\limits_{n\in\mathbb{Z}}\varphi^n 
    \end{array}
    \end{equation}
where the $\widehat{B}$-module structure of $\mathcal{G}(M)$ is given by 
	$$
	\widehat{a}\cdot\widehat{v}:= (a_nv_n+d^{n-1}_M\left(b_{n-1}\otimes v_{n-1}\right))_{n\in\mathbb{Z}},\ \ \text{for all $\widehat{a}=\left(a_n,b_n\right)_{n\in\mathbb{Z}}\in \widehat{B}$ and $\widehat{v}=\left(v_n\right)_{n\in\Z}\in\mathcal{G}(M)$.}
    $$

%------------------------------------------------%

\begin{proposition}
\label{B[A]functors1}
The categories $\widehat{B}\textup{-mod}$ and $\mathfrak{L}^b_{B\otimes_A-}(A\textup{-mod})$ are equivalent abelian $\Bbbk$-categories. The (covariant) equivalence is given by the functor $\mathcal{F}: \widehat{B}\textup{-mod}\longrightarrow \mathfrak{L}^b_{B\otimes_A-}(A\textup{-mod})$ and the functor $\mathcal{G}:\mathfrak{L}^b_{B\otimes_A-}(A\textup{-mod})\longrightarrow \widehat{B}\textup{-mod}$ is a pseudo-inverse.
\end{proposition}

%-----------------------------------------------------------%

\begin{proof}
Let us consider $M\in\widehat{B}$-mod and $g\in\Hom_{\widehat{B}\textup{-mod}}(M,N)$. Hence, 
	\begin{equation*}\label{eq:isom}
	\mathcal{G}\mathcal{F}(M)=\mathcal{G}\left(\left(M^n,d^{n}_M\right)_{n\in\Z}\right)=\bigoplus_{n\in\Z}M^n=\bigoplus_{n\in\Z}\widehat{e}_nM=M.
	\end{equation*}
From Lemma~\ref{lem:PedroGerman} we have the latter equality. Since we can consider any $v\in M$ of the form $v=\left(\widehat{e}_nv_n\right)_{n\in\Z}\in\bigoplus_{n\in\Z}\widehat{e}_nM=M$, it follows that
	\begin{align*}
	\left(\mathcal{G}\mathcal{F}\left(g\right)\right)\left(v\right)
	&=\left(\bigoplus\limits_{n\in\mathbb{Z}}g^n\right)(v)
    =\left(g^{n}\left(\widehat{e}_nv_n\right)\right)_{n\in\Z}=\left(\widehat{e}_ng\left(v_n\right)\right)_{n\in\Z}=\left(g\left(\widehat{e}_nv_n\right)\right)_{n\in\Z}=g(v).
	\end{align*}
Therefore, $\mathcal{G}\mathcal{F}\cong \textup{Id}_{\widehat{B}\textup{-mod}}$ is the identity functor of $\widehat{B}\textup{-mod}$.

On the other hand, consider $M=(M^n, d^n_M)$ and $N=(N^n, d^n_N)$ two objects in $\mathfrak{L}^b_{B\otimes_A-}(A\textup{-mod})$. To facilitate the exposition, we denote $\mathcal{G}(M)$ by $\widehat{M}$. Thus,
	$$
	\mathcal{F}\mathcal{G}(M)=\mathcal{F}(\widehat{M})=\left(\widehat{M}^n,d^{n}_{\widehat{M}}\right)_{n\in\Z}.
	$$
Since $\widehat{e}_n=\left(\delta_{nm},0\right)_{m\in\Z}\in\widehat{B}$ for all $n\in\Z$, from the action of $\widehat{B}$ over $\widehat{M}$ we obtain that
	$$
	\widehat{M}^n= \widehat{e}_n\widehat{M}=\widehat{e}_n\left(\bigoplus_{i\in\Z}M^i\right)=M^n
	$$
and for each $\widehat{a}=\left(a_n,b_n\right)_{n\in\Z}\in\widehat{B}$ and $ v=\left(v_n\right)_{n\in\Z}\in \widehat{M}$, we have
	\begin{align*}
	d^n_{\widehat{M}}\left(\left(\widehat{e}_{n+1}\widehat{a}\widehat{e}_n\right)\otimes\widehat{e}_n v\right)	&=\widehat{e}_{n+1}\widehat{a}\widehat{e}_n v\\
	&= \left(0,\delta_{nm}b_m\right)_{m\in\Z}\cdot\left(v_n\right)_{n\in\Z}\\
	&=\left(d_M^{m-1}\left(\delta_{n,m-1}b_{m-1}\otimes v_{m-1}\right)\right)_{m\in\Z}\\
	&=d^n_M\left(b_n\otimes v_n\right)\\
    &=d^n_M\left(\left(\widehat{e}_{n+1}\widehat{a}\widehat{e}_n\right)\otimes\widehat{e}_n v\right).
	\end{align*}
Therefore, $
	\mathcal{F}\mathcal{G}(M)=\left(\widehat{M}^n,d^n_{\widehat{M}}\right)_{n\in\Z}=\left(M^n,d^n_M\right)_{n\in\Z}=M.
	$
 
Finally, for any morphism $\varphi=(\varphi^n)_{n\in\mathbb{Z}}: M\longrightarrow N$ in $\mathfrak{L}^b_F(A\textup{-mod})$ we denote $\mathcal{G}(\varphi)$ by $\widehat{\varphi}$, hence
	$$
	\mathcal{F}\mathcal{G}\left(\varphi\right)=\mathcal{F}\left(\widehat{\varphi}\right)=\left(\widehat{\varphi}^n\right)_{n\in\Z}=\left(\varphi^n\right)_{n\in\Z}=\varphi
	$$
because,  $\widehat{\varphi}^n\left(v\right)=\widehat{\varphi}^n\left(\widehat{e}_n v\right)=\widehat{e}_n\widehat{\varphi}\left( v\right)=\widehat{e}_n\varphi^n\left(v\right)=\varphi^n\left(v\right)$, for each $n\in\Z$ and for all $v\in M^n$. In conclusion,  $\mathcal{F}\mathcal{G}\cong\textup{Id}_{\mathfrak{L}^b_{B\otimes_A-}(A\textup{-mod})}$, is the identity functor of $\mathfrak{L}^b_{B\otimes_A-}(A\textup{-mod})$, which completes the proof.

\end{proof}

%-----------------------------------------------------------%

%-----------------------------------------------------------%

\begin{remark}\rm
Observe that the functor $\mathcal{G}$ in \eqref{equiv:basic} can be applied in a more general context $\mathcal{G}:\mathfrak{L}^b_{B\otimes_A-}(A\textup{-Mod}) \longrightarrow \widehat{B}\textup{-Mod}$. In this context, Proposition \ref{B[A]functors1} is extended to an equivalence between $\mathfrak{L}^b_{B\otimes_A-}(A\textup{-Mod})$ and the full subcategory of $\widehat{B}\textup{-Mod}$ generated by the image of the extended functor, we will denote by $\widehat{B}\textup{-Mod}_f$. This category coincides with the  category $\widehat{\Lambda}\textup{-Mod}_f$ introduced in \cite[p. 138]{GK} when considering $\Lambda=A\textup{-Mod}$ and $B=D(A)$.
\end{remark}

%-------------------------%

Since $A\otimes_A-=\textup{Id}_{A\textup{-mod}}\in\textup{End}(A\textup{-mod})$ is the identity functor and the fact that  $\mathfrak{L}^b_{\textup{Id}_{A\textup{-mod}}}(A\textup{-mod})=\textup{Ch}^b(A\textup{-mod})$ is  the category of bounded cochain complexes of $A\textup{-mod}$ (Example \ref{example:L_F(C)}(i)), we obtain the following result when we take the $A$-bimodule $B=A$.

%-------------------------------------------------%

\begin{corollary}
The categories $\widehat{A}\textup{-mod}$ and complexes cochain $\textup{Ch}(A\textup{-mod})$ are equivalent. 
\end{corollary}

%-------------------------------------------------%
 
Due to the fact that $\widehat{D(A)}$ is the repetitive algebra of $A$ defined in \cite{HW83}.

%-------------------------------------------------%

\begin{corollary}
The categories $\widehat{D(A)}\textup{-mod}$ and $\mathfrak{L}^b_{D(A)\otimes_A-}(A\textup{-mod})$ are equivalent. 
\end{corollary}
 
%-------------------------------------------------%

Considering the $A$-bimodule $B=\textup{Ext}^2_A(D(A),A)$, we can recover the category of modules over the \textit{cluster repetitive algebra} as follows.

%---------------------------------------%

\begin{corollary}
The category $\mathfrak{L}^b_{B\otimes_A-}(A\textup{-mod})$ with $B=\textup{Ext}^2_A(DA,A)$, is equivalent to the category of modules $\widehat{B}\textup{-mod}$ of the Cluster repetitive algebra $\widehat{B}$ of $A$ 
\end{corollary}
 
%------------------------------%

For details about Cluster repetitive algebra and its category of modules we refer the reader to \cite[\S 3.3]{Ass18}.

%-----------------------------------------------------------%

\begin{remark}\label{rem:repgen}\rm 
Using the method from \cite[Lemma 3.7]{FGRV}, we can transfer key structures and properties from the category $\widehat{D(A)}$-mod (finitely generated modules over the repetitive algebra) to the new category $\widehat{B}$-mod, with $B$ any $A$-bimodule. Specifically, the authors in \cite{FGRV} employ a technique involving the lifting of orthogonal primitive idempotents to obtain projective, injective, and indecomposable modules, along with a class of morphisms associated with them. These elements are lifted from the repetitive algebra $\widehat{D(A)}$ to a new algebra $R\widehat{D(A)}=R\otimes_{\Bbbk}\widehat{D(A)}$, which looks like a repetitive algebra. Here, $R$ denotes a local, complete, commutative, Noetherian $\Bbbk$-algebra with residual field $\Bbbk$. This technique of lifting can be translated word-to-word to the category $\widehat{B}$-mod to obtain similar results over the classes of projective, injective, and indecomposable modules.
\end{remark}

%---------------------%

Clearly, the \textit{Nakayama functor} belong to $\textup{End}_{re}(A\textup{-mod})$ associated to $A$-bimodule $D(A):=\text{Hom}_{\Bbbk}(A,\Bbbk)$. Thus, the general claim is:

%------------------------------------------%

\begin{theorem}
\label{prop:class}
Let $F\in\textup{End}_{re}(A\textup{-mod})$ be a right exact endofunctor. If $F$ is $\otimes$-representable by $B\in A\textup{-bimod}$ then the categories $\mathfrak{L}^b_F(A\textup{-mod})$, $\mathfrak{L}^b_{B\otimes_A -}(A\text{-}\textup{mod})$ and $\widehat{B}\textup{-mod}$ are equivalent.
\end{theorem}

%----------%

\begin{proof}
Since, by hypothesis, $F$ is naturally isomorphic to $B\otimes_A -$, the conclusion is a consequence of the functorial properties of the correspondence described in Subsection \ref{subsec:natfunct}.

\end{proof}

%------------------------------------------%

In the result below, we will use the concept of a \emph{finite $\Bbbk$-category}. Following \cite[Definition 1.8.5]{EGNO15}, we define a finite $\Bbbk$-category as a $\Bbbk$-linear abelian category that is equivalent to the category $A\textup{-mod}$ of left finite-dimensional modules over a finite-dimensional $\Bbbk$-algebra $A$.

\begin{corollary}
\label{cor:KS}
Let $\mathcal{C}$ be a finite $\Bbbk$-category. If $F\in\textup{End}_{re}(\mathcal{C})$ is a right exact endofunctor, then $\mathfrak{L}^b_F(\mathcal{C})$ is a Krull--Schmidt category, with enough projectives and injectives.
\end{corollary}

%------------------------------------------%
\begin{proof}
By definition of finite $\Bbbk$-category and Corollary \ref{cor:Equiv},
there exists a unique right exact endofunctor $F_0$ on $A$-mod, such that we obtain an equivalence of categories $\mathfrak{L}^b_F(\mathcal{C})\simeq \mathfrak{L}^b_{F_0}(A\textup{-mod})$. Thus, by Theorem \ref{prop:class}, Proposition \ref{B[A]functors1} and Remark \ref{rem:repgen}, we obtain that $\mathfrak{L}^b_{F_0}(A\textup{-mod})$ is a Krull--Schmidt category, with enough projectives and injectives, {\it a fortiori}, $\mathfrak{L}^b_F(\mathcal{C})$ satisfies the same properties.
\end{proof}

%------------------------------------------%

\begin{corollary}
\label{cor:Frob}
If $F\in\textup{End}_{re}(A\textup{-mod})$ is a right exact endofunctor $\otimes$-representable by a nonzero $A$-bimodule $I$ which is injective as left and right $A$-module, then $\mathfrak{L}^b_F(A\text{-}\textup{mod})$ is a Frobenius category.

%If $F\in\textup{End}_{re}(A\textup{-mod})$ is a right exact endofunctorwhich is $\otimes$-representable by an  injective $A$-bimodule $I$, then $\mathfrak{L}^b_F(A\text{-}\textup{mod})$ is a Frobenius category.
\end{corollary}

%------------------%

\begin{proof}
By Corollary \ref{cor:KS} we obtain that $\mathfrak{L}^b_F(A\text{-}\textup{mod})$ has enough projectives and injectives. Now, using the natural equivalence of Proposition \ref{prop:proj-inj} and the isomorphism by Proposition \ref{B[A]functors1}, we can lift (see Remark \ref{rem:repgen}) the following two results to the context of the category $\widehat{I}$-mod:
\begin{itemize}
\item \cite[Proposition 6]{Gir18} characterizes the indecomposable projective modules on $\widehat{I}$-mod.
\item \cite[Lemma 2.2]{H88} establishes the equivalence between projective indecomposable modules with injective indecomposable modules on $\widehat{I}$-mod, using the aforementioned equivalence between projective and injective modules.
\end{itemize}
Thus, we conclude that $\mathfrak{L}^b_{I\otimes_A-}(A\text{-}\textup{mod})$ is a Frobenius category.

\end{proof}

%========================%
\subsection{Some geometrical implications}
\label{geom_app}
In this section, following the ideas and results of \cite{BC}, our main result is to obtain a ``global'' analogue of Proposition \ref{prop:class} in the commutative case, as detailed in Theorem \ref{thm:schemes}. These consequences are studied within the context of the symmetric strict monoidal categories of quasi-coherent and coherent sheaves on certain classes of schemes X, which we denote by $\textbf{QCoh}(X)$ and $\textbf{Coh}(X)$, respectively. The term ``global'' is used because if $X$ is an affine algebraic scheme, then $\textbf{QCoh}(X)$ (resp. $\textbf{Coh}(X)$) is equivalent to the category of modules (resp. finitely generated modules) over the commutative $\Bbbk$-algebra over a field $\Bbbk$ of global sections, $A=\mathcal{O}_X(X)$. In a particular application, if $A$ is an Artin $\Bbbk$-algebra (or $X$ is a zero-dimensional affine algebraic scheme), we recover Proposition \ref{prop:class} in this geometric context as an application of Corollary \ref{cor:schemes}. These results can be informally interpreted as establishing that the categories $\mathfrak{L}^b_{F}(\textup{\textbf{QCoh}}(X))$ correspond to ``generalized repetitive algebra'' under certain conditions on $X$ and $F$. This correspondence holds when the pullback is understood as a generalization of a change of base. Given that the applications below use the core ideas from the characterization in \cite{BC}, which connects with affine 2-algebraic geometry, we believe that our results have eventual applications in these contexts.

\begin{theorem}
\label{thm:schemes}
Let $X$ be quasi-compact and quasi-separated scheme and let $F\in\textup{End}(\textup{\textbf{QCoh}}(X))$ be a cocontinuous, symmetric, monoidal endofunctor on the tensor category $\textup{\textbf{QCoh}}(X)$ of quasi-coherent sheaves on $X$. Then there exists an scheme--endomorphism $f:X\longrightarrow X$ such that $\mathfrak{L}^b_{F}(\textup{\textbf{QCoh}}(X))$ is isomorphic to $\mathfrak{L}^b_{f^{\ast}}(\textup{\textbf{QCoh}}(X))$, where $f^{\ast}:\textup{\textbf{QCoh}}(X)\longrightarrow \textup{\textbf{QCoh}}(X)$ is the canonical pullback endofunctor induced by $f$.
\end{theorem}

%----------------------------%

\begin{proof}
By \cite[Theorem 3.4.3, p. 684]{BC} we obtain the analogous equivalence of Proposition \ref{prop:rightrep} in the scheme-theoretical context. The conclusion of theorem is a consequence from functor $\mathfrak{L}^b$.

\end{proof}

%-----------------------------%

\begin{corollary}\label{cor:schemes}
Let $X$ be a noetherian algebraic scheme and $F\in\textup{End}_{re}(\textup{\textbf{Coh}}(X))$ a right exact endofunctor on the category $\textup{\textbf{Coh}}(X)$ of the coherent sheaves. Then there exists an endomorphism $f:X\longrightarrow X$ such that 
$\mathfrak{L}^b_{F}(\textup{\textbf{Coh}}(X))$ is isomorphic to $\mathfrak{L}^b_{f^{\ast}}(\textup{\textbf{Coh}}(X))$, where 
$f^{\ast}:\textup{\textbf{Coh}}(X)\longrightarrow \textup{\textbf{Coh}}(X)$ is the canonical pullback functor induced by $f$. In particular, if $A$ is a commutative noetherian $\Bbbk$-algebra and $F\in\textup{End}_{re}(A\textup{-mod})$ then there exists $f\in\textup{End}_A(A)$ such that $\mathfrak{L}^b_F(A\textup{-mod})$ is isomorphic to $\mathfrak{L}^b_{f^{\ast}}(A\textup{-mod})$, with $f^{\ast}:A\textup{-mod}\longrightarrow A\textup{-mod}$ is the base-change functor induced by $f$, i.e., $M\longmapsto M\otimes_{A,f}A$.
\end{corollary}

%----------------------%

\begin{proof}
Firstly, is well-known that any noetherian algebraic scheme is a quasi-compact and quasi-separated scheme, and $\textbf{QCoh}_{fp}(X)$, the category of quasi-coherent $\mathcal{O}_X$-modules of finite presentation, coincides with the category $\textbf{Coh}(X)$, when $X$ is a noetherian scheme. More details in \cite[Chapter 2]{Hart}. Thus, the conclusion is an immediately consequence of \cite[Corollary 3.4.4, p. 684]{BC} and Theorem \ref{thm:schemes}.

\end{proof}

%------------------%

\begin{corollary}
Let $X$ be quasi-compact and quasi-separated scheme and let $F\in\textup{End}(\textup{\textbf{QCoh}}(X))$ be a cocontinuous, symmetric, monoidal endofunctor on the tensor category $\textup{\textbf{QCoh}}(X)$ of quasi-coherent sheaves on $X$. Then, $\mathfrak{L}^b_{F}(\textup{\textbf{QCoh}}(X))$ is an abelian, complete and co-complete category.
\end{corollary}

%--------------------%

\begin{proof}
It is well-known that the category $\textbf{QCoh}(X)$ is a Grothendieck category. Now, by Theorem \ref{thm:schemes}, $\mathfrak{L}^b_{F}(\textbf{QCoh}(X))$ is isomorphic to the category $\mathfrak{L}^b_{f^{\ast}}(\textbf{QCoh}(X))$, for certain scheme-endomorphism $f:X\longrightarrow X$. Since $f^{\ast}:\textbf{QCoh}(X)\longrightarrow \textbf{QCoh}(X)$, the canonical pullback endofunctor, admits a natural left adjoint endofunctor $f_{\ast}:\textbf{QCoh}(X)\longrightarrow \textbf{QCoh}(X)$ given by the canonical pushforward, the conclusion is a consequence of Corollary \ref{cor:adjconsq}.

\end{proof}

%=============================+===============%
					%ACKNOWLEDGMENTS
%=============================================%

%\vspace{1cm}
%\noindent{\bf Acknowledgments. } 

%The first author was partially supported by the Coordena\c c\~ao de Aperfei\c coamento de Pessoal de N\'ivel Superior - Brasil (CAPES) - Finance Code 001 and also by CODI (Universidad de Antioquia, UdeA) by project numbers 2020-33713 and 2022-52654. The second author was partially supported by CODI (Universidad de Antioquia, UdeA) by project numbers 2020-33713, 2022-52654 and 2023-62291.

%===================================================%
                    %DECLARATIONS%
%===================================================%

\section*{Declarations}

\noindent 
{\bf Acknowledgments.} 
The authors are especially grateful to the anonymous referee for their valuable comments, suggestions, and corrections, which significantly improved the quality of this article.

\vspace{,5cm}

\noindent 
{\bf Ethical Approval.} 
Not applicable.

\vspace{,5cm}

\noindent 
{\bf Funding.} 
The first author was partially supported by the Coordena\c c\~ao de Aperfei\c coamento de Pessoal de N\'ivel Superior - Brasil (CAPES) - Finance Code 001 and also by CODI (Universidad de Antioquia, UdeA) by project numbers 2020-33713 and 2022-52654. The second author was partially supported by CODI (Universidad de Antioquia, UdeA) by project numbers 2020-33713, 2022-52654 and 2023-62291.

\vspace{,5cm}

\noindent 
{\bf Availability of data and materials.} 
Not applicable.

%==========================================%
				%REFERENCES% 
%==========================================%

\end{document}